\documentclass[11pt]{article} 

\usepackage{amsfonts,amsthm,amsmath,amssymb}
\usepackage{color} 
\usepackage{hyperref}
\usepackage{enumitem}
\usepackage{geometry}
\usepackage{cleveref}
\usepackage{a4wide}

\newtheorem{proposition}{Proposition}[section]
\newtheorem{theorem}[proposition]{Theorem}
\newtheorem{lemma}[proposition]{Lemma}
\newtheorem{corollary}[proposition]{Corollary}
\newtheorem{definition}[proposition]{Definition}

\newenvironment{proofof}[1]{\smallskip\noindent{\textbf{Proof~of~#1.}}%
  \hspace{1pt}}{\hspace{-5pt}{\nobreak\quad\nobreak\hfill\nobreak%
    $\square$\vspace{2pt}\par}\smallskip\goodbreak}

\numberwithin{equation}{section}
\allowdisplaybreaks[4]

\newcommand{\LL}[1]{\mathbf{L^#1}}
\renewcommand{\L}[1]{\mathbf{L^#1}}

\newcommand{\dd}[1]{\mathinner{\mathrm{d}{#1}}}
\newcommand{\diver}{\mathinner{\mathop{\rm div}}}
\renewcommand{\div}{\mathinner{\mathop{\rm div}}}
\newcommand{\grad}{\mathinner{\mathop{\rm grad}}}
\newcommand{\tr}{\mathinner{\mathop{\rm tr}}}
\newcommand{\Lloc}[1]{\mathbf{L^{#1}_{loc}}}

\newcommand{\C}[1]{\mathbf{C^{#1}}}

\newcommand{\Cc}[1]{\mathbf{C_c^{#1}}}

\newcommand{\modulo}[1]{{\left|#1\right|}}
\newcommand{\norma}[1]{{\left\|#1\right\|}}
\newcommand{\reali}{{\mathbb{R}}}
\newcommand{\naturali}{{\mathbb{N}}}

\newcommand{\Lip}{\mathop\mathbf{Lip}}

\newcommand{\OO}{\mathinner{\mathcal{O}(1)}}
\newcommand{\sgn}{\mathop\mathrm{sgn}}

\newcommand{\Id}{\mathop{\mathbf{Id}}}

\newcommand{\caratt}[1]{\chi_{\strut#1}}

\newcommand{\eps}{\varepsilon}
\renewcommand{\epsilon}{\varepsilon}

\begin{document}

\title{General Renewal Equations\\ Motivated by Biology and
  Epidemiology}

\author{R.M.~Colombo$^1$ \and M.~Garavello$^2$ \and F.~Marcellini$^1$
  \and E.~Rossi$^3$}

\maketitle

\footnotetext[1]{Universit\`a degli Studi di Brescia, Unit\`a INdAM \&
  Dipartimento di Ingegneria dell'Informazione, via Branze, 38, 25123
  Brescia, Italy.}  \footnotetext[2]{Universit\`a degli Studi di
  Milano Bicocca, Dipartimento di Matematica e Applicazioni, via
  R.~Cozzi, 55, 20125 Milano, Italy.} \footnotetext[3]{Universit\`a
  degli Studi di Modena e Reggio Emilia, Dipartimento di Scienze e
  Metodi dell'Ingegneria, via Amendola, 2, 42122 Reggio Emilia,
  Italy.}

\begin{abstract}

  \noindent We present a unified framework ensuring well posedness and
  providing stability estimates to a class of Initial -- Boundary
  Value Problems for renewal equations comprising a variety of
  biological or epidemiological models. This versatility is achieved
  considering fairly general -- possibly non linear and/or non local
  -- interaction terms, allowing both low regularity assumptions and
  independent variables with or without a boundary. In particular,
  these results also apply, for instance, to a model for the spreading
  of a Covid like pandemic or other epidemics. Further applications
  are shown to be covered by the present setting.


  \medskip

\end{abstract}

\section{Introduction}
\label{sec:Intro}

In a variety of biological models, different species are typically
described through their densities $u^1, u^2,\ldots, u^k$ and, in
general, each $u^h$ depends on time $t \in \reali_+$, on age
$a \in \reali_+$, on a spatial coordinate in $\reali^2$ or $\reali^3$
and possibly also on some structural variables. Thus, a unified
treatment of these models finds its natural setting in the following
general mixed Initial -- Boundary Value Problem (IBVP) in
$\mathcal{X} = \reali_+^m \times \reali^n$
\begin{equation}
  \label{eq:1}
  \!\!\!\!\left\{
    \begin{array}{@{\,}l@{\qquad}r@{\,}c@{\,}l@{}}
      \partial_t u^h
      +
      \diver_x \left(v^h (t,x) \, u^h\right)
      =
      g^h \left(t, x, u(t, x), u(t)\right)
      & (t,x)
      & \in
      & \reali_+ \times\mathcal{X}
      \\
      u^h (t,\xi) = u_b^h\left(t,\xi, u(t)\right)
      & (t,\xi)
      & \in
      & \reali_+ \times \partial\mathcal{X}
      \\
      u^h (0,x) = u_o^h (x)
      & x
      & \in
      & \mathcal{X}\,,
    \end{array}
  \right.
\end{equation}
where $h = 1, \ldots, k$. Aiming at a rather general setting while keeping sharp estimates,
without any loss in generality, we write~\eqref{eq:1} in the form
\begin{equation}
  \label{eq:32}
  \!\!\!\!\left\{
    \begin{array}{@{\,}l@{\qquad}r@{\,}c@{\,}l@{}}
      \partial_t u^h
      +
      \diver_x \left(v^h (t,x) \, u^h\right)
      =
      p^h\left(t, x, u (t)\right) u^h
      + q^h\left(t, x, u, u (t)\right)
      & (t,x)
      & \in
      & I {\times} \mathcal{X}
      \\
      u^h (t,\xi) = u_b^h\left(t,\xi, u(t)\right)
      & (t,\xi)
      & \in
      & I {\times} \partial\mathcal{X}
      \\
      u^h (0,x) = u_o^h (x)
      & x
      & \in
      & \mathcal{X}\,,
    \end{array}
  \right.
\end{equation}
where $h = 1, \ldots, k$. Note that the decomposition of the source
term $g^h$ in~\eqref{eq:1} into $p^h$ and $q^h$ is neither unique nor
in any sense restrictive.

We stress that both in~\eqref{eq:1} and in~\eqref{eq:32} the term
$u (t)$ appearing in the right hand sides is understood as a
\emph{function}, so that both the source and boundary terms
in~\eqref{eq:1}, besides being \emph{non linear}, also comprise quite
general \emph{non local}, i.e., \emph{functional}, dependencies.

The current literature comprehends a multitude of well known models
fitting into~\eqref{eq:1}: we recall here for
instance~\cite{MR2496711, MR4263205, BELL1967329, MR3013074, preprint,
  KangRuan2021, LorenziEtAl, MeeardTran_2009, PerthameBook}, leaving
to Section~\ref{sec:App} the highlighting of specific aspects
of~\eqref{eq:1} in other recent or classical models. In particular,
the well posedness and stability theorems below apply also to
model~\eqref{eq:4} which, to our knowledge, does not fully fit into
other well posedness results in the literature. At the same time, the
literature covering particular instances of~\eqref{eq:1} dates back to
classical milestones, such as~\cite{MR354068, KermackMcKendrick1927,
  Lotka339, m'kendrick_1925}. Moreover, various textbooks introduce to
the analytical study of models fitting into~\eqref{eq:1}, see for
instance~\cite{MR3700352, MR3616174, MR3887640, MR860962,
  PerthameBook, MR772205}.

A multitude of compartmental models share the key features of the
chosen framework~\eqref{eq:1}: they are the domain $\mathcal X$ of the
$x$ variable and the coexistence of rather general local and non local
terms. Indeed, under the choice of $\mathcal{X}$ above, we comprise
also bounded space/age domains~\cite{KangRuan2021}, half
lines~\cite{FisterEtAl2004}, full vector spaces~\cite{LorenziEtAl} as
well as their combinations~\cite{BELL1967329, preprint,
  NordmannPerthameTaing2017, TuckerZimmerman1988}. In all these cases,
rather general conditions are assigned along the different types of
boundaries that fit into~\eqref{eq:1}, such as, for instance, natality
terms~\cite{BELL1967329, NordmannPerthameTaing2017,
  TuckerZimmerman1988}
. The biological meaning imposes that these boundary terms, as well as
the sources in~\eqref{eq:1}, may contain both local and non local
terms. The former ones comprehend, for
instance, 
mortality terms~\cite{MR3013074, preprint}, while the latter can be
motivated by natality~\cite{BELL1967329, NordmannPerthameTaing2017},
predation~\cite{Elena2015} or interaction between
populations~\cite{MR3013074}, e.g., the propagation of an
infection~\cite{preprint}.

We underline that the present framework does not rely on any
regularizing effect of diffusion. The general non local terms here
considered need not have any smoothing effect, and can also be
absent. The lack of diffusion operators ensures that any movement or
evolution described by~\eqref{eq:1} propagates with a \emph{finite}
speed. In particular, the present approach is consistent with
deterministic modeling, while the Laplace operator may also serve to
describe various sorts of random effects, see for
instance~\cite{AinsebaIannelli2003, LanglaisBusenberg1997}.

\smallskip

Within this general framework, we first prove well posedness, i.e.,
local existence, uniqueness and continuous dependence of the solution
to~\eqref{eq:1} on the initial datum. Then, we provide conditions
ensuring the global in time existence and the stability with respect
to functions and parameters defining~\eqref{eq:1}. Throughout, the
functional setting is provided by $\LL1$ and the distance between
solutions is always evaluated through the $\LL1$ norm. As a
consequence, we can deal with non smooth solutions, a necessary
feature in view of control problems. Moreover, the boundedness neither
of the total variation nor of the $\LL\infty$ norm of the data is
required. Indeed, among the different notions of solutions to IBVPs
for renewal equations, we choose to establish our framework on that
introduced in~\cite{Martin, Vovelle}. This definition not only is
stated in terms of integral inequalities, more convenient in any
limiting procedure, but remarkably it does not require any notion of
trace, allowing us to deal with merely $\LL1$ solutions.

Remark that in~\eqref{eq:1} both the source terms and the boundary
terms are non linear. Thus, a key tool in the proofs is Banach
Contraction Theorem, based on precise estimates on scalar
equations. Merely requiring some sort of local Lipschitz regularity
does not rule out the possibility of finite time blow ups (in any
norm), as shown below by explicit examples. We thus resort to a
Gronwall type argument to obtain global in time existence. As a
byproduct, we also record a uniqueness result in the general setting
of~\eqref{eq:1} based, as in the classical {Kru\v zkov} case, on a
carefully chosen definition of solution, see\
\S~\ref{sec:defin-semi-entr}.

We also note that particular instances of equations falling
within~\eqref{eq:1} can be studied through other techniques, such as,
for instance, analytic semigroup theory, generalized entropy methods
or Laplace transform. We refer, for instance, to~\cite{MR3700352,
  MR3616174, MR3887640, PerthameBook}.

\smallskip

The present results, besides unifying the treatment of various models,
provide tools useful in tackling control/optimization problems based
on~\eqref{eq:1}. Indeed, the stability estimates proved in
Theorem~\ref{thm:stab} ensure that general integral functional defined
on the solutions are Lipschitz continuous functions of the data and
parameters characterizing~\eqref{eq:1}. A further direction that can
be pursued using the present results is that of inverse problems,
i.e., exhibiting conditions ensuring that an optimal choice of data
and parameters in~\eqref{eq:1} is possible, in order to best fit sets
of given experimental data.

\smallskip

This paper is organized as follows.  In
Section~\ref{sec:assumptions-results} we provide the basic well
posedness and stability results.  Then, Section~\ref{sec:App} is
devoted to specific applications that fit into~\eqref{eq:1}. The
technical analytic proofs are deferred to the final
Section~\ref{sec:AP}.

\section{Assumptions, Definitions and Results}
\label{sec:assumptions-results}

Throughout, we set $\reali_+ = \mathopen[0, +\infty\mathclose[$,
\begin{equation}
  \label{eq:33}
  I = \reali_+
  \quad \mbox{ or } \quad
  I = [0,T]
  \quad \mbox{ and } \quad
  \mathcal{X} = \reali_+^m \times \reali^n
\end{equation}
for a positive $T$.

First, we state what we mean by \emph{solution} to~\eqref{eq:1}. To
this aim, we extend to the present case the definitions
in~\cite{Martin, Vovelle}, see in
particular~\cite[Definition~3.5]{ElenaBoundary2018}.

\begin{definition}
  \label{def:sol}
  A map $u_* \in \C0(I; \LL1 (\mathcal{X}; \reali^k))$ is a
  \emph{solution} to~\eqref{eq:1} if setting for $h = 1, \ldots, k$,
  $t \in I$, $x \in \mathcal{X}$ and $\xi \in \partial\mathcal{X}$
  \begin{displaymath}
    \mathcal{G}^h (t,x)
    =
    g^h\left(
      t,x, u_*(t, x), u_*(t)
    \right)
    \quad \mbox{ and }\quad
    \mathcal{U}_b^h (t,\xi)
    =
    u_b^h\left(t, \xi, u_*(t) \right) \,,
  \end{displaymath}
  for $h = 1, \ldots,k$ the map $u_*^h$ is a semi--entropy solution to
  the IBVP
  \begin{displaymath}
    \left\{
      \begin{array}{l@{\qquad}r@{\,}c@{\,}l@{}}
        \partial_t u
        +
        \diver_x \left(v^h (t,x) \, u\right)
        =
        \mathcal{G}^h (t,x)
        & (t,x)
        & \in
        & I \times \mathcal{X}
        \\
        u (t,\xi) = \mathcal{U}_b^h (t,\xi)
        & (t,\xi)
        & \in
        & I \times \partial\mathcal{X}
        \\[2pt]
        u (0,x) = u_o^h (x)
        & x
        & \in
        & \mathcal{X} \,.
      \end{array}
    \right.
  \end{displaymath}
\end{definition}

\noindent We recall in Definition~\ref{def:mvsol} below the notion of
semi-entropy solution.


The main result of this paper concerns the well posedness of the
Cauchy Problem~\eqref{eq:32}.

\begin{theorem}
  \label{thm:main}
  Use the notation~\eqref{eq:33} and let the following assumptions
  hold:

  \begin{enumerate}[label=\bf{(V)}, ref=\textup{\textbf{(V)}}, align =
    left]
  \item \label{ip:(v)}
    $v \in (\C1 \cap \LL\infty)(I \times \mathcal{X}; \reali^{k\times
      (n+m)})$,
    $\diver_x v^h \in \Lloc1 (I; \LL\infty(\mathcal{X}; \reali))$ for
    $h = 1, \ldots, k$ and there exists a positive $V$ such that
    \begin{displaymath}
      \left(v^h (t,x)\right)_i > V
      \qquad \forall \, (t,x) \in I \times \partial\mathcal{X}
      \mbox{ and for }
      \begin{array}{r@{\,}c@{\,}l@{}}
        h
        & =
        &1, \ldots, k\,;
        \\
        i
        & =
        & 1, \ldots, m \,.
      \end{array}
    \end{displaymath}
  \end{enumerate}

  \begin{enumerate}[label=\bf{(P)}, ref=\textup{\textbf{(P)}},
    align=left]
  \item \label{hyp:g_a} For all $w \in \LL1(\mathcal{X}; \reali^k)$,
    the map $(t,x) \to p (t,x,w)$ is in
    $\C0 (I\times\mathcal{X}; \reali^k)$ and there exist positive
    $P_1$ and $P_2$ such that for $t \in I$, $x \in \mathcal{X}$,
    $w, w' \in \LL1(\mathcal{X}; \reali^k)$
    \begin{eqnarray*}
      \norma{p(t, x, w)}
      & \leq
      & P_1 + P_2 \, \norma{w}_{\LL1(\mathcal{X};\reali^k)} \,;
      \\
      \norma{p(t, x, w) - p(t, x, w')}
      & \leq
      & P_2 \, \norma{w - w'}_{\LL1(\mathcal{X};\reali^k)} \,.
    \end{eqnarray*}
  \end{enumerate}

  \begin{enumerate}[label=\bf{(Q)}, ref=\textup{\textbf{(Q)}},
    align=left]
  \item \label{hyp:g_b} For all $w \in \LL1(\mathcal{X}; \reali^k)$,
    the map $(t,x,u) \to q (t,x,u,w)$ is in
    $\C0 (I \times \mathcal{X}\times \reali^k; \reali^k)$ and there
    exist positive $Q_1$ and $Q_3$ and a function
    $Q_2 \in (\LL1 \cap \LL\infty)(\mathcal{X}; \reali_+)$ such that
    for $t \in I$, $x \in \mathcal{X}$, $u, u' \in \reali^k$,
    $w, w' \in \LL1(\mathcal{X}; \reali^k)$:
    \begin{eqnarray*}
      \norma{q(t, x, u, w)}
      & \leq
      & Q_1 \, \norma{u}
        + Q_2(x) \, \norma{w}_{\LL1(\mathcal{X};\reali^k)}
        + Q_3 \, \norma{u} \, \norma{w}_{\LL1(\mathcal{X};\reali^k)} \,;
      \\
      \norma{q(t, x, u, w) - q(t, x, u', w')}
      & \leq
      &
        Q_1 \, \norma{u - u'}
        + Q_3 \, \norma{w}_{\LL1(\mathcal{X};\reali^k)} \, \norma{u - u'}
      \\
      &
      & \quad
        + \, Q_3 \, \norma{u'} \, \norma{w - w'}_{\LL1(\mathcal{X};\reali^k)} \,.
    \end{eqnarray*}
  \end{enumerate}

  \begin{enumerate}[label=\bf{(BD)}, ref=\textup{\textbf{(BD)}}, align
    = left]
  \item\label{ip:(ub)}
    $u_b \colon \reali_+ \times \partial \mathcal{X} \times
    \LL1(\mathcal X; \reali^k) \to \reali^k$ is such that for any
    $w \in \LL1 (\partial\mathcal{X}; \reali^k)$, the map
    $(t,\xi) \to u_b(t, \xi, w)$ is measurable. Moreover, there exists
    a function
    $B \in (\LL1 \cap \LL\infty) (\partial \mathcal{X}; \reali_+)$
    such that for every $t \in I$, $\xi \in \partial\mathcal{X}$,
    $w,w' \in \LL1(\mathcal{X}; \reali^k)$,
    \begin{eqnarray*}
      \norma{u_b(t, \xi, w)}
      & \leq
      & B(\xi)
        \left(
        1
        +
        \norma{w}_{\LL1(\mathcal{X},\reali^k)}
        \right)
      \\
      \norma{u_b(t, \xi, w) - u_b(t, \xi, w')}
      & \leq
      & B(\xi) \, \norma{w - w'}_{\LL1(\mathcal{X},\reali^k)} \,.
    \end{eqnarray*}
  \end{enumerate}

  \begin{enumerate}[label=\bf{(ID)}, ref=\textup{\textbf{(ID)}}, align
    = left]
  \item \label{ip:(u0)} $u_o \in \LL1 (\mathcal{X}; \reali^k)$.
  \end{enumerate}

  \noindent Then,
  \begin{enumerate}[label=\bf{(WP.\arabic*)},
    ref=\textup{\bf{(WP.\arabic*)}}, align = left]

  \item \label{thm:it:1} There exists a positive $T_* \in I$ such
    that, setting $I_* = [0,T_*]$, the IBVP~\eqref{eq:32} admits a
    solution in the sense of \Cref{def:sol} defined on $I_*$.

  \item \label{thm:it:uni} Assume $u_1$ and $u_2$ solve~\eqref{eq:32}
    in the sense of Definition~\ref{def:sol} with
    $u_1, u_2 \in \LL\infty (I\times\mathcal{X}; \reali^k)$. Then,
    $u_1 = u_2$.

  \item \label{thm:it:ID} Let
    $\hat u_o, \check u_o \in \LL1 (\mathcal{X};\reali^k)$. If
    $\hat u \colon \hat I \to \reali^k$, respectively
    $\check u \colon \check I \to \reali^k$, solve~\eqref{eq:32} in
    the sense of Definition~\ref{def:sol} with initial datum
    $u_o = \hat u_o$, respectively $u_o = \check u_o$, then there
    exists a function
    $\mathcal{L} \in \Lloc\infty (\hat I \cap \check I; \reali)$ such
    that for all $t \in \hat I \cap \check I$
    \begin{displaymath}
      \norma{\hat{u}(t) - \check{u}(t)}_{\LL1 (\mathcal{X};\reali^k)}
      \leq
      \mathcal{L} (t) \;
      \norma{\hat{u}_{o} - \check{u}_{o}}_{\LL1 (\mathcal{X};\reali^k)} \,.
    \end{displaymath}
  \end{enumerate}
\end{theorem}

\noindent The proof is deferred to Section~\ref{sec:AP}.

In several applications it is of interest to guarantee that each
component in the solution attains non negative values. To this aim, we
state the following Corollary.

\begin{corollary}
  \label{cor:piu}
  Let the same assumptions of Theorem~\ref{thm:main} hold and assume
  moreover that for an index $h \in \{1, \ldots, k\}$
  \begin{enumerate}[label=\bf{(Q+)}, ref=\textup{\textbf{(Q+)}},
    align=left]
  \item \label{hyp:g-positivity} For $t \in I$,
    a.e.~$x \in \mathcal{X}$, $u \in \reali_+^k$,
    $w \in \LL1(\mathcal{X}; \reali_+^k)$, $q^h(t, x, u, w) \ge 0$.
  \end{enumerate}
  \begin{enumerate}[label=\bf{(BD+)}, ref=\textup{\textbf{(BD+)}},
    align = left]
  \item \label{ip:(ub-2)} For $t \in I$, $\xi \in \partial \mathcal X$
    and $w \in \LL1(\mathcal{X}; \reali^k)$,
    $u_b^h\left(t, \xi, w\right) \ge 0$.
  \end{enumerate}
  \begin{enumerate}[label=\bf{(ID+)}, ref=\textup{\textbf{(ID+)}},
    align = left]
  \item \label{ip:(u0plus)} For a.e. $x \in \mathcal{X}$,
    $u_o^h (x) \geq 0$.
  \end{enumerate}
  \noindent Then the unique solution $u$ to~\eqref{eq:32} also
  satisfies for every $t \in I_*$ and for a.e. $x \in \mathcal{X}$.
  \begin{equation}
    \label{eq:positivity}
    u^h(t, x) \ge 0 \,.
  \end{equation}
\end{corollary}

\noindent The proof is deferred to Section~\ref{sec:AP}.

The above result is of a local nature and, without further
assumptions, it can not be extended to a global result, as the
following examples show. Consider the Cauchy Problem~\eqref{eq:32}
with $k=1$, $m=0$, $n=1$, $\mathcal{X} = \reali$,
$p (t,x,w) = \int_0^1 w (x) \dd{x}$, $q \equiv 0$, which results in
\begin{displaymath}
  \left\{
    \begin{array}{l}
      \partial_t u = u \, \int_0^1u (t,x) \dd{x}
      \vspace{.2cm}\\
      u (0,x) = \caratt{[0,1]} (x)
    \end{array}
  \right.
  \quad \mbox{ solved by } \quad
  u (t,x) = \dfrac{1}{1-t} \; \caratt{[0,1]} (x) \,.
\end{displaymath}
Note that~\ref{hyp:g_a} holds with $P_1=0$ and $P_2=1$. Clearly, $u$
blows up in any norm at $t=1$.

Similarly, setting $k=1$, $m=1$, $n=0$, $\mathcal{X} = \reali_+$,
$p (t,x,w) = \int_{\reali_+} w (x) \dd{x}$, $q \equiv 0$
in~\eqref{eq:32}, which satisfies~\ref{hyp:g_a} with $P_1=0$ and
$P_2=1$, leads to the Cauchy Problem
\begin{displaymath}
  \left\{
    \begin{array}{l}
      \partial_t u + \partial_x u = u \, \int_{\reali_+} u (t,x) \dd{x}
      \\
      u (t, 0) = 0
      \\
      u (0,x) = \caratt{[0,1]} (x) \,,
    \end{array}
  \right.
  \quad \mbox{ solved by } \quad
  u (t,x) = \dfrac{1}{1-t} \; \caratt{[t,t+1]} (x) \,.
\end{displaymath}
Again, the solution blows up in any norm at $t=1$.

Typical biological/epidemiological models have further properties
ensuring that solutions are defined globally in time. In particular,
the model described in \S~\ref{subs:AP} displays a quadratic right
hand side similar to those in the examples above, differing in the
sign. Nevertheless, in this example, well posedness holds globally in
time. Indeed, in general, a lower bound on the solutions is available
since Corollary~\ref{cor:piu} ensures that the components of the
solution attain non negative values. An upper bound, preventing finite
time blow up, is obtained through assumption~\ref{ip:(ub)} on the
boundary datum and a further condition, see~\eqref{eq:36} below, that
bounds the overall growth.

\begin{corollary}
  \label{cor:assumpt-defin-result}
  Let $I = \reali_+$.  Let the assumptions of Corollary~\ref{cor:piu}
  hold for all $h=1, \ldots, k$. Assume moreover that for suitable
  $C_1 \in \Lloc\infty (\reali_+ ; \LL1(\mathcal{X}; \reali))$ and
  $C_2 \in \Lloc\infty (\reali_+; \reali)$,
  \begin{equation}
    \label{eq:36}
    \sum_{h=1}^kp^h (t,x,w) \, u^h + q^h (t,x,u,w)
    \leq C_1(t,x) + C_2 (t) \, \sum_{h=1}^k u^h
  \end{equation}
  for all $t \in \reali_+$, a.e. $x \in \mathcal{X}$,
  $u,w \in \reali^k$. Then, the solution to~\eqref{eq:32} is defined
  for all $t \in \reali_+$.
\end{corollary}

Finally, we provide the stability estimates essential to tackle, for
instance, control problems. To this aim, we need to slightly
specialize the functional dependence of $p$, $q$ and $u_b$ on $u
(t)$. We thus obtain sufficient conditions to apply
Theorem~\ref{thm:main} and get stability estimates.

\begin{theorem}
  \label{thm:stab}
  Let assumptions~\ref{ip:(v)} and~\ref{ip:(u0)} hold. Assume that
  in~\eqref{eq:32}, for $t \in I$, $x \in \mathcal{X}$,
  $u \in \reali^k$, $w \in \LL1 (\mathcal{X}; \reali^k)$,
  \begin{equation}
    \label{eq:37}
    \begin{array}{rcl}
      p^h (t,x,w)
      & =
      & P^h \left(t, x, \int_{\mathcal{X}} \mathcal{K}_p^h (t,x,x') \, w (x') \dd{x'}\right)
      \\
      q^h (t,x,u,w)
      & =
      & Q^h \left(t, x, u, \int_{\mathcal{X}} \mathcal{K}_q^h (t,x,x') \, w (x') \dd{x'}\right)
      \\
      u_b^h (t,\xi,w)
      & =
      & U_b^h\left(
        t, \xi, \int_{\mathcal{X}} \mathcal{K}_u^h (t,\xi,x') \, w (x') \dd{x'}
        \right) \,,
    \end{array}
  \end{equation}
  where the functions above satisfy:
  \begin{enumerate}[label=\bf{($\overline{\mathbf{P}}$)},
    ref=\textup{\bf($\overline{\mathbf{P}}$)}, align = left]
  \item \label{stab:it:1} There exist $\bar P_1 \ge 0$ and
    $\bar P_2 \ge 0$ such that, for every $h=1, \ldots, k$, the
    function $P^h: I \times \mathcal X \times \reali^{k_p} \to \reali$
    ($k_p \ge 1$) satisfies
    \begin{displaymath}
      \modulo{P^h\left(t, x, \eta\right)} \le
      \bar P_1 + \bar P_2 \norma{\eta}
      \quad \mbox{ and } \quad
      \modulo{P^h\left(t, x, \eta_1\right) - P^h\left(t, x, \eta_2\right)}
      \le \bar P_2 \norma{\eta_1 - \eta_2}
    \end{displaymath}
    for every $t \in I$, $x \in \mathcal X$,
    $\eta, \eta_1, \eta_2 \in \reali^{k_p}$;
    $\mathcal K_p^h \in \LL\infty(I \times \mathcal X^2; \reali^{k_p
      k})$.
  \end{enumerate}
  \begin{enumerate}[label=\bf{($\overline{\mathbf{Q}}$)},
    ref=\textup{\bf($\overline{\mathbf{Q}}$)}, align = left]
  \item \label{stab:it:2} There exist $\bar Q_1, \bar Q_3 \ge 0$ and
    $\bar Q_2 \in \left(\LL1 \cap \LL\infty\right) \left(\mathcal X;
      \reali^+\right)$ such that, for every $h=1, \ldots, k$, the
    function
    $Q^h: I \times \mathcal X \times \reali^k \times \reali^{k_p} \to
    \reali^+$ ($k_q \ge 1$) satisfies
    \begin{eqnarray*}
      \modulo{Q^h\left(t, x, u, \eta\right)}
      & \!\!\!\!\!\le\!\!\!\!\!
      & \bar Q_1\norma{u} + \bar Q_2(x) \norma{\eta}
        + \bar Q_3 \norma{u} \norma{\eta}
      \\
      \modulo{Q^h\left(t, x, u_1, \eta_1\right) {-} Q^h\left(t, x, u_2,
      \eta_2\right)}
      & \!\!\!\!\!\le\!\!\!\!\!
      & \bar Q_1 \norma{u_1 - u_2} + \bar Q_3 \norma{\eta_1}
        \norma{u_1 - u_2} + \bar Q_3 \norma{u_2} \norma{\eta_1 - \eta_2}
    \end{eqnarray*}
    for every $t \in I$, $x \in \mathcal X$,
    $u, u_1, u_2 \in \reali^k$,
    $\eta, \eta_1, \eta_2 \in \reali^{k_q}$;
    $\mathcal K_q^h \in \LL\infty(I \times \mathcal X^2; \reali^{k_q
      k})$.
  \end{enumerate}
  \begin{enumerate}[label=\bf{($\overline{\mathbf{BD}}$)},
    ref=\textup{\bf($\overline{\mathbf{BD}}$)}, align = left]
  \item \label{stab:it:3} There exists
    $\bar B \in (\LL1 \cap \LL\infty) (\partial\mathcal{X}; \reali_+)$
    such that for every $h=1, \ldots, k$, the function
    $U_b^h \colon I \times \partial\mathcal{X} \times \reali^{k_u} \to
    \reali_+$ satisfies
    \begin{displaymath}
      \modulo{U_b^h (t, \xi, \eta)}
      \leq
      \bar B (\xi)  \left(1 + \norma{\eta}\right)
      \quad \mbox{ and } \quad
      \modulo{U_b^h (t, \xi, \eta_1) - U_b^h (t, \xi, \eta_2)}
      \leq
      \bar B (\xi) \, \norma{\eta_1 - \eta_2}
    \end{displaymath}
    for every $t \in I$, $\xi \in \partial\mathcal{X}$ and
    $\eta, \eta_1, \eta_2 \in \reali^{k_u}$;
    $\mathcal K_u^h \in \LL\infty(I \times \partial \mathcal X \times
    \mathcal X; \reali^{k_u k})$.
  \end{enumerate}
  Then, Theorem~\ref{thm:main} applies. Moreover, if both systems
  \begin{eqnarray}
    \label{eq:25}
    \!\!\!\!\left\{
    \begin{array}{@{\,}l@{\quad}r@{\,}c@{\,}l}
      \partial_t u^h
      +
      \diver_x \left(v^h (t,x) \, u^h\right)
      =
      \hat p^h\left(t, x, u(t)\right) u^h
      + \hat q^h\left(t, x, u, u(t)\right)
      & (t,x)
      & \in
      & I {\times}\mathcal{X}
      \\
      u^h (t,\xi) = \hat{u}_b^h\left(t,\xi, u(t)\right)
      & (t,\xi)
      & \in
      & I {\times} \partial\mathcal{X}
      \\
      u^h (0,x) = \hat{u}_o^h (x)
      & x
      & \in
      & \mathcal{X}\,,
    \end{array}
        \right.
    \\
    \label{eq:26}
    \!\!\!\!\left\{
    \begin{array}{@{\,}l@{\quad}r@{\,}c@{\,}l}
      \partial_t u^h
      +
      \diver_x \left(v^h (t,x) \, u^h\right)
      =
      \check p^h\left(t, x, u(t)\right) u^h
      + \check q^h\left(t, x, u, u(t)\right)
      & (t,x)
      & \in
      & I {\times}\mathcal{X}
      \\
      u^h (t,\xi) = \check{u}_b^h\left(t,\xi, u(t)\right)
      & (t,\xi)
      & \in
      & I {\times} \partial\mathcal{X}
      \\
      u^h (0,x) = \check{u}_o^h (x)
      & x
      & \in
      & \mathcal{X}\,,
    \end{array}
        \right.
  \end{eqnarray}
  satisfy the assumptions above, then the following stability
  estimates hold:
  \begin{eqnarray*}
    &
    & \norma{\hat{u}(t) - \check{u}(t)}_{\LL1 (\mathcal{X};\reali^k)}
    \\
    & \leq
    & \OO
      \bigl[
      \norma{\hat{P} - \check{P}}_{\LL\infty\left([0,t] \times \mathcal X
      \times \reali^{k_p}; \reali^k\right)}
      +
      \norma{\hat{\mathcal{K}}_p - \check{\mathcal{K}}_p}_{\LL\infty ([0,t]\times\mathcal{X}^2; \reali^{k_p k^2})}
      \bigr.
    \\
    &
    &
      \qquad\qquad
      +
      \norma{\hat Q - \check Q}_{\LL1 ([0,t]\times\mathcal{X}; \LL\infty (\reali^k\times\reali^{k_q}; \reali^k
      ))}
      +
      \norma{\hat{\mathcal{K}}_q - \check{\mathcal{K}}_q}_{\LL\infty ([0,t]\times\mathcal{X}^2; \reali^{k_q k^2})}
    \\
    &
    & \qquad\qquad
      \bigl.
      + \norma{\hat U_b - \check U_b}_{\LL1 ([0,t]\times \partial\mathcal{X}; \LL\infty (\reali^{k_u};\reali^k))}
      +
      \norma{\hat{\mathcal{K}}_u - \check{\mathcal{K}}_u}_{\LL\infty ([0,t]\times\partial\mathcal{X}\times\mathcal{X}; \reali^{k_u k^2})}
      \bigr]
      e^{\OO t}
  \end{eqnarray*}
  for every $t$ such that $\hat u$ and $\check u$ are defined on
  $[0,t]$ and where the Landau symbol $\OO{}$ denotes a constant
  independent of the initial data.
\end{theorem}

\noindent The proof is deferred to Section~\ref{sec:AP}.

Finally, we note that~\ref{ip:(v)} and Definition~\ref{def:sol} allow
to immediately extend all results in the present section to the case
$\mathcal{X} = \left(\prod_{i=1}^m I_i\right) \times \reali^n$, as
soon as $I_1, \ldots, I_m$ are (non trivial) real intervals bounded
below. In particular, any of the $I_i$ may well be bounded also above.

\subsection{The Definition of Semi--Entropy Solution Ensures
  Uniqueness}
\label{sec:defin-semi-entr}

This paragraph provides a definition of solution and the consequent
uniqueness statement in a setting more general than the one usually
found in the literature. In particular, it extends the results
in~\cite[Section~3]{Martin} to the slightly more general case of the
unbounded domain $\mathcal{X}$. Indeed, with the
notation~\eqref{eq:33}, consider the fully nonlinear IBVP
\begin{equation}
  \label{eq:14}
  \left\{
    \begin{array}{l@{\qquad}r@{\,}c@{\,}l}
      \partial_t u + \diver_x f (t,x,u)
      =
      g (t,x,u)
      & (t,x)
      & \in
      & I \times \mathcal{X}
      \\
      u (t, \xi) = u_b (t,\xi)
      & (t,\xi)
      & \in
      & I \times \partial\mathcal{X}
      \\
      u (0,x) = u_o (x)
      & x
      & \in
      & \mathcal{X} \,.
    \end{array}
  \right.
\end{equation}
The following definition is the extension to~\eqref{eq:14}
of~\cite[Definition~3.5]{ElenaBoundary2018}, see also~\cite{Martin,
  Vovelle}.

\begin{definition}
  \label{def:mvsol}
  A \emph{semi-entropy solution} to the IBVP~\eqref{eq:14} on the real
  interval $I$ is a map
  $u \in \Lloc\infty (I; \LL1( \mathcal{X}; \reali))$ such that for
  any $\kappa \in \reali$ and for any test function
  $\varphi \in \Cc1 (\mathopen]-\infty,\sup I\mathclose[ \times
  \reali^{n+m}; \reali_+)$
  \begin{align}
    \nonumber
    & \int_I \int_{\mathcal{X}}
      \left(u (t,x) - \kappa \right)^\pm
      \partial_t \varphi (t,x)
      \dd{x} \dd{t}
    \\
    \nonumber
    & + \int_I \int_{\mathcal{X}}
      \sgn{}^\pm (u (t,x) -\kappa)
      \left(f (t,x, u) - f (t,x,\kappa)\right) \cdot
      \grad_x \varphi (t,x)
      \dd{x} \dd{t}
    \\
    \label{eq:mv}
    & + \int_I \int_{\mathcal{X}}
      \sgn{}^\pm (u (t,x) -\kappa)
      \left[
      g\left(t,x,u (t,x) \right) - \diver_x f (t,x,\kappa)
      \right] \, \varphi (t,x) \dd{x}
      \dd{t}
    \\ \nonumber
    & + \int_{\mathcal{X}}
      \left(u_o (x) -
      \kappa \right)^\pm
      \, \varphi (0,x)
      \dd{x}
    \\
    \nonumber
    & + \Lip (f)
      \int_I \int_{\partial \mathcal{X}}
      \left(u_b(t,\xi) - \kappa \right)^\pm \, \varphi
      (t,\xi) \dd\xi \dd{t} \geq 0
  \end{align}
  where $\Lip (f)$ is a Lipschitz constant of the map
  $u \to f (t,x,u)$, uniform in $(t,x) \in I \times \mathcal{X}$.
\end{definition}

\noindent Above, we use the notation $w^+ = \max \{w, 0\}$ and
$w^- = \max \{-w, 0\}$.

A key feature of~\eqref{eq:mv} is its ensuring uniqueness, which we
detail in the next Proposition to ease comparisons with the current
literature.

\begin{proposition}
  \label{prop:Uniqueness}
  Consider the general scalar IBVP~\eqref{eq:14} under the assumptions
  \begin{enumerate}[label=\bf{(f)}, ref=\textup{\textbf{(f)}}, align =
    left]
  \item \label{ip:(f)}
    $f \in \C0 (I \times \mathcal{X} \times \reali; \reali^{n+m})$
    admits continuous derivatives $\partial_u f$,
    $\partial_u \grad_x f$, $D^2_{xx}f$ with $\partial_u f$ and
    $\grad_x f$ bounded in $(t, x) \in I \times \reali_+$ locally in
    $u \in \reali$; $\partial_u \grad_x f$ is bounded.
  \end{enumerate}
  \begin{enumerate}[label=\bf{(g)}, ref=\textup{\textbf{(g)}}, align =
    left]
  \item \label{ip:(lin-G)}
    $g, \partial_u g, \partial_{x_i}g \in \C0 (I\times
    \mathcal{X}\times \reali; \reali)$ and for all
    $(t,x) \in I \times \mathcal{X}$, $\modulo{g (t,x,u)} \leq G (u)$
    for a map $G \in \Lloc\infty (\reali; \reali_+)$ and
    $\partial_u g$ is bounded.
  \end{enumerate}
  \begin{enumerate}[label=\bf{(bd)}, ref=\textup{\textbf{(bd)}}, align
    = left]
  \item \label{ip:(gs-ub)} The boundary datum satisfies
    $u_b \in \LL\infty (I \times \partial\mathcal{X}; \reali)$.
  \end{enumerate}

  \begin{enumerate}[label=\bf{(id)}, ref=\textup{\textbf{(id)}}, align
    = left]
  \item \label{ip:(gs-u0)} The initial datum satisfies
    $u_o \in \LL\infty (\mathcal{X}; \reali)$.
  \end{enumerate}
  \noindent If $u_1,u_2 \in \LL\infty (I\times\mathcal{X}; \reali)$
  both satisfy~\eqref{eq:mv}, then they coincide.
\end{proposition}

This Proposition slightly extends~\cite[Theorem~18]{Martin}. However,
its proof relies on merely technical modifications to~\cite[Lemma~16
and Lemma~17]{Martin}, due to the present unboundedness of the domain
$\mathcal{X}$. Very similar techniques are employed also
in~\cite[\S~2.6 and \S~2.7]{MalekEtAlBook}, which is devoted to a
hyperplane.

\section{Sample Applications}
\label{sec:App}

The structure of~\eqref{eq:1} is sufficiently flexible to comprise a
variety of applications of mathematics to biology, in particular to
epidemiology. The general results in the preceding section can be
applied to well known models in the literature, see for
instance~\cite{MR2496711, MR3013074, ColomboGaravello2017_ControlBio,
  PerthameBook}. In the next paragraphs, we select sample applications
based on analytic structure that differ in the number of equations, in
the number of independent variables, in the presence of (partial)
boundaries and in the role of non local terms. In particular,
\S~\ref{subs:SP} deals with a recently proposed model,
see~\cite{preprint}, while the subsequent ones refer to other
classical models that fit into~\eqref{eq:1}.

\subsection{The Spreading of an Epidemic}
\label{subs:SP}

During the spreading of an epidemic, within a population we
distinguish among individuals that are Susceptible, Infective,
Hospitalized or Recovered, see~\cite{preprint}. Each of these
populations is described through its time, age and space dependent
density: $S = S (t,a,y)$, $I = I (t,a,y)$, $H = H(t,a,y)$ and
$R = R (t,a,y)$, respectively. Remark that the distinction between $I$
and $H$ consists in the $H$ individuals that, being hospitalized or
quarantined, do not infect anyone although being ill. In its most
general form, the model presented in~\cite[\S~2]{preprint} to describe
the evolution of these populations, reads
\begin{equation}
  \label{eq:4}
  \!\!\!
  \left\{
    \begin{array}{c@{\,}c@{\,}c@{\,}c@{\,}c@{\,}c@{\,}c@{\,}%
      c@{\,}r@{\,}c@{\,}c@{\,}c@{\,}c@{\,}c@{\,}c@{\,}}
      \partial_t S
      & +
      & \partial_a S
      & +
      & \diver_y (v_S\, S)
      & +
      & \mu_S \, S
      & =
      & - (\rho \otimes I) S
      \\
      \partial_t I
      & +
      & \partial_a I
      & +
      & \diver_y (v_I\, I)
      & +
      & \mu_I \, I
      & =
      & (\rho \otimes I) S
      & -
      & \kappa \, I
      & -
      & \vartheta \, I
      \\
      \partial_t H
      & +
      & \partial_a H
      &
      &
      & +
      & \mu_H \, H
      & =
      &
      & +
      & \kappa \, I
      &
      &
      & -
      & \eta \, H
      \\
      \partial_t R
      & +
      & \partial_a R
      & +
      & \diver_y (v_R\, R)
      & +
      & \mu_R \, R
      & =
      &
      &
      &
      & +
      & \vartheta\, I
      & +
      & \eta \, H
    \end{array}
  \right.
  \qquad\qquad
  \begin{array}{r@{\,}c@{\,}l@{}}
    t
    & \in
    & \reali_+
    \\
    a
    & \in
    & \reali_+
    \\
    y
    & \in
    & \reali^2
  \end{array}
\end{equation}
where the propagation of the infection is described by
\begin{equation}
  \label{eq:6}
  \left(\rho \otimes I (t)\right) (a,y) =
  \int_{\reali_+} \int_{\reali^2}
  \rho (a, a', y, y') \, I (t,a', y') \, \dd{y'} \, \dd{a'} \,.
\end{equation}
Here, the function $\rho$ plays the key role of describing how
infective individuals infect others, at which distance and with which
dependence on age or time, see~\cite{preprint} for more
details. In~\eqref{eq:4}, $v_S = v_S (t,a,y)$, $v_I = v_I (t,a,y)$ and
$v_R = v_R (t,a,y)$ describe the time, age and, possibly, space
dependent movements of the $S$, $I$ and $R$ individuals, while
$\mu_S = \mu_S (t,a,y)$, $\mu_I = \mu_I (t,a,y)$,
$\mu_H = \mu_H (t,a,y)$ and $\mu_R = \mu_R (t,a,y)$ are the
mortalities. The term $\kappa = \kappa (t,a,y)$ describes how quickly
infected individuals are confined to quarantine;
$\vartheta = \vartheta (t,a,y)$, respectively $\eta = \eta (t,a,y)$,
quantifies the speed at which infected, respectively quarantined,
individuals recover.

System~\eqref{eq:4} needs to be supplemented by boundary and initial
data:
\begin{equation}
  \label{eq:7}
  \left\{
    \begin{array}{r@{\,}c@{\,}l}
      S (t,a=0, y)
      & =
      & S_b (t,y)
      \\
      I (t,a=0, y)
      & =
      & 0 
      \\
      H (t,a=0, y)
      & =
      & 0 
      \\
      R (t,a=0, y)
      & =
      & 0 
    \end{array}
  \right.
  \qquad \mbox{ and } \qquad
  \left\{
    \begin{array}{r@{\,}c@{\,}l}
      S (t=0,a,y)
      & =
      & S_o (a,y)
      \\
      I (t=0,a,y)
      & =
      & I_o (a,y)
      \\
      H (t=0,a,y)
      & =
      & H_o (a,y)
      \\
      R (t=0,a,y)
      & =
      & R_o (a,y) \,.
    \end{array}
  \right.
\end{equation}
Note that a more precise boundary term, though not amenable to be used
in the short term, might be a natality term of the form
\begin{displaymath}
  S (t,a=0,y)
  =
  \int_{\reali_+} b (t,a',y) \, S (t,a',y) \dd{a'}
\end{displaymath}
which also fits in the framework of Theorem~\ref{thm:main} and
Theorem~\ref{thm:stab}. Note
that~\eqref{eq:4}--\eqref{eq:6}--\eqref{eq:7} is a system with
independent variables $(a,y)$ where $a$ is bounded below while $y$ is
in $\reali^2$ and no second order differential operator is
present. The model~\eqref{eq:4}--\eqref{eq:6}--\eqref{eq:7} fits
into~\eqref{eq:32} in the form~\eqref{eq:37} setting
$\mathcal{X} = \reali_+ \times \reali^2$, $x = (a,y)$, $\xi = (0,y)$
and
\begin{displaymath}
  \begin{array}{@{}r@{\;}c@{\;}l@{\qquad}r@{\;}c@{\;}l@{\qquad}%
    r@{\;}c@{\;}l@{\qquad}r@{\,}c@{\;}l@{}}
    k
    & =
    & 4
    & m
    & =
    & 1
    & n
    & =
    & 2
    \\
    u^1
    & =
    &S
    & u^2
    & =
    & I
    & u^3
    & =
    & H
    & u^4
    & =
    & R
    \\
    w^1
    & =
    &S (t)
    & w^2
    & =
    & I (t)
    & w^3
    & =
    & H (t)
    & w^4
    & =
    & R (t)
    \\
    v^1
    & =
    & \left[
      \begin{array}{@{}c@{}}
        1\\v_S
      \end{array}
    \right]
    & v^2
    & =
    & \left[
      \begin{array}{@{}c@{}}
        1\\v_I
      \end{array}
    \right]
    & v^3
    & =
    & \left[
      \begin{array}{@{}c@{}}
        1\\0
      \end{array}
    \right]
    & v^4
    & =
    & \left[
      \begin{array}{@{}c@{}}
        1\\v_R
      \end{array}
    \right]
    \\
    u^1_b
    & =
    & S_b
    & u^2_b
    & =
    & 0
    & u^3_b
    & =
    & 0
    & u^4_b
    & =
    & 0
    \\
    u_o^1
    & =
    & S_o
    & u_o^2
    & =
    & I_o
    & u_o^3
    & =
    & H_o
    & u_o^4
    & =
    & R_o
  \end{array}
\end{displaymath}
\begin{displaymath}
  \begin{array}{rcl}
    p^1(t,x,\Lambda)
    & =
    & - \mu_S - \Lambda
    \\
    p^2(t,x,\Lambda)
    & =
    & - \mu_I - \kappa - \vartheta
    \\
    p^3(t,x,\Lambda)
    & =
    & - \mu_H - \eta
    \\
    p^4(t,x,\Lambda)
    & =
    & - \mu_R
  \end{array}
  \qquad
  \begin{array}{rcl}
    q^1(t, x, u,\Lambda)
    & =
    & 0
    \\
    q^2(t, x, u,\Lambda)
    & =
    & \Lambda \; u_1
    \\
    q^3(t, x, u,\Lambda)
    & =
    & \kappa \, u_2
    \\
    q^4(t, x, u,\Lambda)
    & =
    & \vartheta \, u_2 + \eta \, u_3
  \end{array}
\end{displaymath}
and the only $2$ non zero entries in $\mathcal{K}_p$ and
$\mathcal{K}_q$ are valued $\rho$, so that
\begin{eqnarray*}
  \int_{\mathcal{X}}
  \mathcal{K}_p^1 \left(t,(a,y),(a',y')\right)
  \; w (a',y') \dd{a'} \dd{y'}
  & =
  & \left(\rho \otimes I (t)\right) (a,y) \,,
  \\
  \int_{\mathcal{X}}
  \mathcal{K}_q^2 \left(t,(a,y),(a',y')\right)
  \; w (a',y') \dd{a'} \dd{y'}
  & =
  & \left(\rho \otimes I (t)\right) (a,y) \,.
\end{eqnarray*}

\begin{proposition}
  \label{prop:spreading-pandemic}
  Set $\mathcal{I} = [0,T]$ or $\mathcal{I} = \reali_+$.  Let
  $v_S, v_I, v_R \in (\C1 \cap \LL\infty) (\mathcal{I} \times
  \mathcal{X}; \reali^2)$ with divergence in
  $\LL1 (\mathcal{I}; \LL\infty (\mathcal{X};\reali))$;
  $\rho \in \LL\infty (\reali_+^2 \times \reali^4; \reali)$ and
  $S_b \in (\LL1 \cap \LL\infty) (\mathcal{I} \times \reali^2;
  \reali)$. Let $\mu_S$, $\mu_I$, $\mu_H$, $\mu_R$, $\vartheta$,
  $\eta$ and $\kappa$ be positive and in $\LL\infty$. Fix an initial
  datum $(S_o, I_o, H_o, R_o)$ in $\LL1 (\mathcal{X};
  \reali^4)$. Then:
  \begin{enumerate}
  \item Problem~\eqref{eq:4}--\eqref{eq:6}--\eqref{eq:7} fits into
    Theorem~\ref{thm:main} and Theorem~\ref{thm:stab} and hence admits
    a solution
    $(S,I,H,R) \in \C0\left([0,T_*]; \L1 (\mathcal{X};
      \reali^4)\right)$, for a $T_*>0$.

  \item If the initial and boundary data $(S_o, I_o, H_o, R_o)$ and
    $S_b$ are non negative, if $\rho\geq 0$ and if the constants
    $\kappa, \eta, \theta$ are non negative, then
    Corollary~\ref{cor:piu} applies, ensuring that the solution is non
    negative: $(S,I,H,R) (t) \in \L1 (\mathcal{X}; \reali_+^4)$, for
    all $t \in [0,T_*]$.

  \item If, in addition to what required at 2., the mortalities
    $\mu_S, \mu_I, \mu_H, \mu_R$ are non negative,
    then~\Cref{cor:assumpt-defin-result} applies, so that the solution
    is defined globally in time.

  \item If, in addition to what required at 3., $(S_o, I_o, H_o, R_o)$
    in $\L\infty (\mathcal{X}; \reali_+^4)$, then the solution is
    locally bounded:
    $(S,I,H,R) \in \L\infty (\mathcal{J} \times \mathcal{X};
    \reali_+^4)$, for any bounded interval
    $\mathcal{J} \subseteq \mathcal{I}$. Hence, $(S,I,H,R)$ is the
    unique solution to~\eqref{eq:4} in the sense of
    Definition~\ref{def:sol}.
  \end{enumerate}
\end{proposition}

\noindent The proof is deferred to Section~\ref{sec:AP}.

As pointed out in~\eqref{eq:4}, a natural control parameter is the
coefficient $\kappa = \kappa (t,a,y)$, which determines how quickly
infective individuals are isolated in quarantine.

A first natural choice for a \emph{cost} to be minimized by a careful
choice of $\kappa$ is the total number of deaths on the time interval
$[0,T]$, namely
\begin{displaymath}
  \mathcal{D} (\kappa)
  =
  \int_0^T \! \int_{\reali_+} \! \int_{\reali^2}
  \left(\mu_I (t,a,y) \, I (t,a,y) + \mu_H  (t,a,y) \, H (t,a,y)\right)
  \, \dd{y} \, \dd{a} \, \dd{t} \,.
\end{displaymath}
Proposition~\ref{prop:spreading-pandemic} ensures that the cost
$\mathcal{D}$ is a continuous function of $\kappa$. Hence, standard
compactness arguments, for instance in the case of a constant
$\kappa$, ensure the existence of an optimal control. Moreover, the
Lipschitz continuity, again ensured by
Proposition~\ref{prop:spreading-pandemic}, allows to use standard
optimization algorithms to actually find near--to--optimal controls.

A second reasonable choice is to minimize the maximal number of
infected individuals
$\norma{I}_{\L\infty ([0,T] \times \reali_+ \times \reali^2)}$, aiming
at minimizing the maximal stress on the health care system. Again, the
continuity proved in Proposition~\ref{prop:spreading-pandemic} allows
to use Weierstrass type arguments to exhibit the existence of optimal
controls, thanks to the lower semicontinuity of the $\L\infty$ norm
with respect to the $\L1$ distance.

\subsection{Cell Growth and Division}
\label{subs:GandD}

Consider the classical model~\cite[Formula~(2)]{BELL1967329} devoted
to the description of cell growth and cell division, as extended
in~\cite[Formul\ae~(1.5)--(1.7)]{TuckerZimmerman1988}:
\begin{equation}
  \label{eq:22}
  \left\{
    \begin{array}{l}
      \partial_t N + \partial_a N + \diver_y (V (a,y)\, N)
      =
      -
      \lambda(a,y) \, N
      \\
      N (t,0,y)
      =
      \int_{\reali_+}\int_{\reali^n}
      \beta \left((a',y'),y, N (t,a',y')\right)
      \dd{y'} \dd{a'}
    \end{array}
  \right.
\end{equation}
where $t \in \reali_+$ is time, $a \in \reali_+$ is age,
$(y_1, \ldots, y_n) \in \reali^n$ is an $n$--tuple of structure
variables, $\lambda = \lambda (a,y)$ is the age-- and state--specific
loss rate, $N = N (t,a,y)$ is the population density and $V = V (a,y)$
is the (time independent) individual cell's growth rate. Therefore,
\eqref{eq:22} fits into~\eqref{eq:32} setting
\begin{displaymath}
  \begin{array}{@{}c@{}}
    k = 1
    \,,\quad
    n \in \naturali
    \,,\quad
    m = 1
    \,,\quad
    \mathcal{X} = \reali_+ {\times} \reali^n
    \,,\quad
    x = (a, y)
    \,,\quad
    \xi = (0,y)
    \,,\quad
    u = N
    \,,\quad
    w = N (t)
    \,,
    \\[6pt]
    v \left(t,(a,y)\right)
    =
    V (a,y)
    \,,\quad
    p\left(t, (a,y), N (t)\right)
    =
    -\lambda (a,y)
    \,,\quad
    q\left(t, (a,y), N, N (t)\right)
    =
    0\,,
    \\[6pt]
    u_b(t,y,N, N (t))
    =
    \displaystyle\int_{\reali^n} \int_{\reali_+}
    \beta\left((a', y'), y, N (t, a', y')\right) \dd{a'} \dd{y'} \,.
  \end{array}
\end{displaymath}
Concerning the assumptions of Theorem~\ref{thm:main}, we have
that~\ref{ip:(v)} is satisfied as soon as
$V \in (\C1 \cap \LL\infty)(\mathcal{X}; \reali^n)$ and
$\diver V \in \LL1 (I; \LL\infty (\mathcal{X};
\reali))$. Condition~\ref{hyp:g_a} is met whenever
$\lambda \in \C0 \cap \LL\infty$, with
$P_1 = \norma{\lambda}_{\LL\infty(\reali_+ \times \reali^n; \reali)}$
and $P_2=0$. Assumption~\ref{hyp:g_b} trivially holds. To comply
with~\ref{ip:(ub)}, we need $\beta$ to be Lipschitz continuous and
sublinear in its fourth argument, i.e.,
$\beta ((a',y') , y, w) \leq B (y) \left(1+ \modulo{w}\right)$ for a
suitable $B \in \LL1 \cap \LL\infty$. Under these assumptions,
Theorem~\ref{thm:main} applies to~\eqref{eq:22}.

As soon as $\beta \geq 0$ and the initial datum is non negative, also
Corollary~\ref{cor:piu} applies, ensuring the solution is non
negative. It is reasonable to assume from the biological point of view
that $\lambda \geq 0$, so that also
Corollary~\ref{cor:assumpt-defin-result} applies (with $C_1 = 0$,
$C_2 = 0$), ensuring that the solution is globally defined in time.
It is straightforward to see that, as soon as $\beta$ is linear in its
third argument, it is possible to apply also Theorem~\ref{thm:stab}.


\subsection{An Age and Phenotypically Structured Population Model}
\label{subs:AP}

Within the general form~\eqref{eq:1} we recover also the recent
model~\cite[Formula~(1)]{NordmannPerthameTaing2017}, namely
\begin{equation}
  \label{eq:12}
  \left\{
    \begin{array}{@{\,}l@{}}
      \displaystyle
      \varepsilon\, \partial_t M_\varepsilon
      +
      \partial_a \left(A (a,y) \, M_\varepsilon\right)
      =
      -\left(
      \int_{\reali_+} \!\int_{\reali^n} \!
      M_\varepsilon (t,a',y') \dd{a'} \dd{y'}
      +
      d (a,y)
      \right) M_\varepsilon
      \\
      \displaystyle
      M_\varepsilon (t,a=0,y)
      =
      \dfrac{1}{A (a=0,y) \, \varepsilon^n}
      \int_{\reali_+} \! \int_{\reali^n}
      \mathcal{M} \! \left(\frac{y'-y}{\varepsilon}\right) \,
      b (a',y') \,
      M_\varepsilon (t, a', y')
      \dd{a'} \dd{y'}
      \\
      M_\varepsilon (t=0, a,y)
      =
      M_\varepsilon^0 (a,y) \,.
    \end{array}
  \right.
\end{equation}
Here, the dependent variable $M_\varepsilon = M_\varepsilon (t,a,y)$
describes the population density at time $t$, of age $a \in \reali_+$
and trait $x \in \reali^n$, so that
$\int_{\reali_+} \!\int_{\reali^n} M_\varepsilon (t,a,y) \dd{a}
\dd{x}$ is the total population. The growth function $A = A (a,y)$
describes the age and trait dependent aging. The mortality, on the
right hand side of the first equation in~\eqref{eq:12}, both depends
on the crowding, due to intraspecies competition, and on a given
mortality $d = d (a,y)$. The function $b = b (a,y)$ quantifies the
natality and is modulated by the mutation probability kernel
$\mathcal{M}$, both defining the boundary term along $a=0$, see
also~\cite{MischlerPerthameRyzhik_2002}.

Note that the IBVP~\eqref{eq:12} can be seen as a prototype equation
for various other similar models, see for
instance~\cite[Formula~(2.8)]{MeeardTran_2009}
.

\smallskip

The above system~\eqref{eq:12} fits into~\eqref{eq:32} setting
$\mathcal{X} = \reali_+ \times \reali^n$ and
\begin{equation}
  \label{eq:13}
  \begin{array}{c}
    k=1 \,,\quad
    m=1 \,,\quad
    n \geq 1\,,\quad
    x = (a,y) \,,\quad
    \xi = (0,y) \,,\quad
    u = M_\eps \,,\quad
    w = M_\eps (t) \,,\quad
    \\
    \displaystyle
    v = \left[
    \begin{array}{@{}c@{}}
      A (a,y)/\eps
      \\
      0
    \end{array}
    \right]
    \,,\qquad
    p (t, x, w)
    =
    - \dfrac{1}{\eps}
    \int_{\reali^n} w(x) \dd{x} - \dfrac{d (x)}{\eps}
    \,,\qquad
    q (t, x, u, w)
    =
    0\,,
    \\
    \displaystyle
    u_b (t, y, w)
    =
    \dfrac{1}{A (a=0,y) \, \varepsilon^n}
    \int_{\reali_+} \! \int_{\reali^n}
    \mathcal{M} \! \left(\frac{y'-y}{\varepsilon}\right) \,
    b (a',y') \,
    w(a', y')
    \dd{a'} \dd{y'} \,.
  \end{array}
\end{equation}

\begin{proposition}
  \label{prop:an-age-phen}
  Let $A \in (\C1 \cap \LL\infty)(\mathcal{X}; \reali)$ with
  $\inf A >0$ and $\div_{a,y}A \in \LL\infty (\mathcal{X};
  \reali)$. Let $d \in \LL\infty (\reali^n; \reali)$,
  $\mathcal{M} \in \LL\infty (\reali^n; \reali)$ such that
  $\mathcal{M} (\eta) = 0$ whenever $\norma{\eta} \geq r$, for a fixed
  $r>0$. Moreover,
  $b \in \LL\infty (\reali_+ \times \reali^n; \reali)$ such that
  $\modulo{b (a,y)} \leq \left(1+\norma{y}\right)^{-(n+1)}$. Then, for
  any initial datum
  $u_o \in (\LL1 \cap \LL\infty)(\mathcal{X}; \reali)$,
  Theorem~\ref{thm:main} applies to the Cauchy Problem
  for~\eqref{eq:12} with datum $u_o$. If moreover $u_o \geq 0$,
  $A (0,y) \geq 0$, $\mathcal{M}\geq0$ and $b\geq 0$,
  Corollary~\ref{cor:piu} and Corollary~\ref{cor:assumpt-defin-result}
  apply, ensuring that the solution is non negative and defined on all
  $\reali_+$.
\end{proposition}

\noindent The proof is deferred to Section~\ref{sec:AP}.  Thus, the
above result ensures existence on $\mathopen[0, +\infty \mathclose[$
as soon as all the assumptions are available therein, recovering the
well posedness results in~\cite{MischlerPerthameRyzhik_2002,
  NordmannPerthameTaing2017}.

\subsection{Further Applications}
\label{sec:furth-appl}

We briefly recall here further models considered in the literature
that fit within~\eqref{eq:1}. In each of the cases below, we refer to
the original sources for detailed descriptions of the modeling
environments.

\smallskip

The model presented in~\cite[Formula~(5)]{LorenziEtAl}, devoted to the
modeling of leukemia development, reads (here, $i = 2, \ldots, M-1$
for a fixed $M \in \naturali$, $M \geq 3$):
\begin{equation}
  \label{eq:9}
  \!\!\!\!\!\!
  \left\{
    \begin{array}{@{}l@{}}
      \partial_t n_1
      =
      \left(\dfrac{2\, a_1 (x)}{1 + K \int_0^1 n_M (t,x') \dd{x'}}-1\right)
      p_1 (x) \, n_1
      \\
      \partial_t n_i
      =
      2 \left(
      1
      {-}
      \dfrac{a_{i-1} (x)}{1 {+} K \! \int_0^1 n_M (t,x') \dd{x'}}
      \right)
      \!  p_{i-1} (x)  \, n_{i-1}
      {+}
      \left(
      \dfrac{2 a_i (x)}{1 {+} K \! \int_0^1 n_M (t,x') \dd{x'}}
      {-}
      1
      \right)
      \!  p_i (x)  \, n_i
      \\
      \partial_t n_M
      =
      2 \left(1-\dfrac{a_{M-1} (x)}{1 + K \int_0^1 n_M (t,x') \dd{x'}}\right)
      p_{M-1} (x) \, n_{M-1} - d\, n_M
      \\
      n_i (0,x)
      =
      n_i^o (x)\,.
    \end{array}
  \right.
  \!\!\!
\end{equation}
Remark that~\eqref{eq:9} can be seen as a system of ordinary
differential equations on functions defined on $[0,1]$ or,
alternatively, as a system of ordinary differential equations coupled
also through a non local dependence on the $x$ variable. Nevertheless,
it fits within~\eqref{eq:1}: 
indeed, set $k = M$, $m=0$, $n=1$, $\mathcal{X} = \reali$,
$u = (n_1, \ldots, n_M)$, $v \equiv 0$, the other terms being
obviously chosen.

It is worth noting that the recent model~\cite[Formula~(13)]
{MR4263205}, though devoted to an entirely different scenario, is
analytically analogous to~\eqref{eq:9} and also fits within the
framework formalized in Section~\ref{sec:assumptions-results}.  The
use of Theorem~\ref{thm:main} and Theorem~\ref{thm:stab} thus extends
the results in~\cite{MR4263205, LorenziEtAl} comprehending $\LL1$
solutions and providing a full set of stability estimates.

\medskip

Another example is the model recently presented
in~\cite[Formula~(1.1)]{KangRuan2021}, devoted to an age--structured
population described by the time, age and space dependent density
$u = u (t,a,y)$:
\begin{equation}
  \label{eq:38}
  \left\{
    \begin{array}{l}
      \partial_t u + \partial_a u = d (J*u (t) - u) + G\left(u (t)\right)
      \\
      u (t,0,y) =  F\left(u (t)\right)
      \\
      u (0,a,y) = \Phi (a,y)
    \end{array}
  \right.
\end{equation}
considered in~\cite{KangRuan2021} for $a \in [0, a^+]$ and
$y \in \Omega$, where $a^+ \in \, \left]0, +\infty\right[$ and
$\Omega \subseteq \reali^N$ are given. Above, $J$ is a convolution
kernel, while the functionals $F$ and $G$ are locally Lipschitz
continuous with respect to the $\LL1$ norm. Model~\eqref{eq:38} fits
into~\eqref{eq:1} setting $k=1$, $m=1$, $n=N$,
$\mathcal{X} = \reali_+ \times \reali^N$, $x = (a,y)$,
$\setlength{\delimitershortfall}{2pt}v = \left[
  \begin{array}{@{}c@{}}
    1\\0
  \end{array}
\right]$, the choice of the other terms being immediate. The results
in Section~\ref{sec:assumptions-results} immediately apply even if the
age interval $[0,a^+]$ and the space domain are bounded, thanks to the
generality of the assumptions required on $v$. This allows to have
qualitative information on the dependence of the solutions exhibited
in~\cite{KangRuan2021} on the various parameters and functions
defining~\eqref{eq:38}.

\medskip

We recall also the following competitive population model with age
structure as an example of a system of equations. It was introduced
and studied from the optimal management point of view
in~\cite[Formula~(1.1)]{FisterEtAl2004}:
\begin{equation}
  \label{eq:39}
  \left\{
    \begin{array}{l}
      \displaystyle
      \partial_t u^1 + \partial_a u^1
      =
      - \mu_1 (a, u^1) \, u^1 - f^1 (t,a)\, u^1
      - u^1 \int_0^A c_1 (a',a)\, u^2 (t,a') \dd{a'}
      \\
      \displaystyle
      \partial_t u^2 + \partial_a u^2
      =
      - \mu_2 (a,u^2)\, u^2 - f^2 (t,a)\, u^2
      - u^2 \int_0^A c_2 (a',a)\, u^1 (t,a') \dd{a'}
      \\
      \displaystyle
      u^1 (t,0) = \int_0^A \beta_1 (a') \, u^1 (t,a') \dd{a'}
      \\
      \displaystyle
      u^2 (t,0) = \int_0^A \beta_2 (a') \, u^2 (t,a') \dd{a'}
      \\
      u^1 (0,a) = u^1_o (a)
      \\
      u^2 (0,a) = u^2_o (a) \,.
    \end{array}
  \right.
\end{equation}
Here, we have $k=2$, $m=1$, $n=0$, $\mathcal{X}= \reali_+$,
$v=1$. Under the assumptions of Theorem~\ref{thm:main} and
Theorem~\ref{thm:stab} we recover the continuity of the profit
functional~\cite[Formula~(1.2)]{FisterEtAl2004}
\begin{displaymath}
  J (f) =
  \int_0^T \int_0^A \left(
    K_1 (a) \, f^1 (t,a)\, u^1 (t,a)
    +
    K_2 (a) \, f^2 (t,a)\, u^2 (t,a)
  \right) \dd{a} \dd{t} \,,
\end{displaymath}
now also in the setting of $\LL1$ solutions.

\section{Analytic Proofs}
\label{sec:AP}

\subsection{The Scalar Case}

We now consider in detail the affine scalar case, namely~\eqref{eq:14}
with $f (t, x, u) = v (t, x) \; u$ and
$g (t,x,u) = p (t,x) \, u + q (t,x)$, i.e.,
\begin{equation}
  \label{eq:2}
  \left\{
    \begin{array}{l@{\qquad}r@{\,}c@{\,}l@{}}
      \partial_t u
      +
      \diver_x \left(v (t,x) u\right)
      =
      p (t,x) \, u + q (t,x)
      & (t,x)
      & \in
      & \reali_+ \times \mathcal{X}
      \\
      u (t,\xi) = u_b (t,\xi)
      & (t,\xi)
      & \in
      & \reali_+ \times \partial\mathcal{X}
      \\
      u (0,x) = u_o (x)
      & x
      & \in
      & \mathcal{X} \,.
    \end{array}
  \right.
\end{equation}

\noindent Recall the following standard notation. A
\emph{characteristic} of~\eqref{eq:2} is the solution
$t \to X (t; t_o,x_o)$ to the following Cauchy Problem for the system
of ordinary differential equations
\begin{equation}
  \label{eq:3}
  \left\{
    \begin{array}{l}
      \dot x = v (t,x)
      \\
      x (t_o) = x_o \,.
    \end{array}
  \right.
  \qquad
  \begin{array}{r@{\,}c@{\,}l}
    (t, x)
    & \in
    & I \times \mathcal{X}
    \\
    (t_o, x_o)
    & \in
    & I \times \mathcal{X} \,.
  \end{array}
\end{equation}
For $\tau, t \in I$ and for $x \in \mathcal{X}$, define
\begin{equation}
  \label{eq:10}
  \mathcal{E} (\tau,t,x)
  =
  \exp\left(
    \int_\tau^t
    \left(
      p\left(s, X (s;t,x)\right)
      -
      \diver_x v \left(s,X (s;t,x)\right)
    \right)
    \dd{s}
  \right)
\end{equation}
and for all $(t,x) \in I \times \mathcal{X}$, if
$x \in X(t; [0,t[, \partial\mathcal{X})$, we set
\begin{equation}
  \label{eq:11}
  T (t,x)
  =
  \inf \left\{
    s \in [0,t[
    \colon
    X (s; t,x) \in \mathcal{X}
  \right\} \,.
\end{equation}
With the notation introduced above, we recall the well known formula
\begin{equation}
  \label{eq:8}
  \!\!\!\!\!
  u (t,x)=
  \left\{
    \begin{array}{l@{\qquad}r@{\,}c@{\,}l@{}}
      u_o \left(X (0;t,x)\right) \mathcal{E} (0,t,x)
      \\
      \displaystyle
      \qquad\qquad
      +
      \int_0^t q\left(\tau, X (\tau;t,x)\right)\, \mathcal{E} (\tau,t,x)\dd\tau
      & x
      & \in
      & X (t; 0,\mathcal{X})
      \\[6pt]
      u_b\left(T (t,x), X\left(T (t,x); t, x\right)\right) \,
      \mathcal{E}\left(T (t,x), t, x\right)
      \\
      \displaystyle
      \qquad\qquad
      + \int_{T (t,x)}^t q\left(\tau, X (\tau; t,x)\right) \, \mathcal{E} (\tau,t,x) \dd\tau
      & x
      & \in
      & X (t; \mathopen[0,t\mathclose[, \partial \mathcal{X})
    \end{array}
  \right.
  \!\!\!
\end{equation}
obtained from the integration along characteristics, a standard tool
at least since the classical paper~\cite{MR354068}. The following
relations are of use below, for a proof see for
instance~\cite[Chapter~3]{BressanPiccoliBook},
\begin{align}
  \label{eq:ptX}
  \partial_t X (t; t_o, x_o)
  & =
    v\left(t, X (t; t_o, x_o)\right)
  \\
  \label{eq:ptoX}
  \partial_{t_o} X (t; t_o, x_o)
  & =
    -v (t_o, x_o) \;
    \exp \int_{t_o}^t \diver_x v\left(s; X(t, t_o, x_o)\right) \dd{s}
  \\
  \label{eq:21}
  D_{x_o} X (t; t_o, x_o)
  & =
    M (t)
    \mbox{, the matrix } M \mbox{ solves }
    \left\{
    \begin{array}{l}
      \dot M = D_xv\left(t, X (t; t_o, x_o)\right) M
      \\
      M (t_o) = \Id \,.
    \end{array}
  \right.
\end{align}

In order to prove that~\eqref{eq:8} solves~\eqref{eq:2} in the sense
of Definition~\ref{def:mvsol} and to provide the basic well posedness
estimates, a few technical lemmas are in order. First introduce the
following notation: where misunderstandings might arise, we use the
positional notation for derivatives. For instance, with reference to
the map $(t; t_o, x_o) \to X (t; t_o, x_o)$, we denote
\begin{displaymath}
  \partial_2 X (t; t_o, x_o)
  =
  \partial_{t_o} X (t; t_o, x_o)
  =
  \lim_{\tau \to 0} \dfrac{X (t;t_o+\tau,x_o) - X (t; t_o, x_o)}{\tau} \,.
\end{displaymath}
We also set $X = (X_1, \ldots, X_{m+n})$, with $X_i = X \cdot e_i$,
where $(e_1, \ldots, e_{m+n})$ is the canonical base of
$\reali^{m+n}$. Recall also that
$\partial_l X_i = \partial_l (X \cdot e_i) = (\partial_l X) \cdot
e_i$, for $l = 1,2,3$ and $i = 1, \ldots, m+n$.

\begin{lemma}
  \label{lem:t-new}
  Under assumption~\ref{ip:(v)} with $k=1$, the map in~\eqref{eq:11}
  \begin{equation}
    \label{eq:16-new}
    \begin{array}{@{}c@{\,}c@{\,}ccc@{}}
      T
      & \colon
      & \left\{ (t,x) \in \reali_+ \times \mathcal{X}
        \colon
        x \in X (t; \left[0, t\right[,\partial\mathcal{X})\right\}
      &\to
      & \reali_+
      \\
      &
      & (t,x)
      & \mapsto
      & \inf \left\{
        s \in [0,t[
        \colon
        X (s; t,x) \in \mathcal{X}
        \right\}
    \end{array}
  \end{equation}
  is well defined. Moreover, for all $t \in \reali_+$ and
  a.e.~$x \in \mathcal{X}$ such that
  $x \in X (t;[0,t[,\partial \mathcal{X})$, there exists a unique
  $i \in \{1, \ldots, m\}$, depending on $t$ and $x$, such that
  \begin{equation}
    \label{eq:15-new}
    X_i (T (t,x); t, x) = 0.
  \end{equation}
  Given $t \in \reali_+$, for $i \in \{1, \ldots, m\}$, call
  ${\mathbb X}_i^t$ the set of $x \in \mathcal{X}$ such that $i$ is
  the \emph{unique} index satisfying~\eqref{eq:15-new}. Then, the map
  \begin{equation}
    \label{eq:20-new}
    \begin{array}{cccc}
      M_i \colon
      & {\mathbb X}_i^t
      & \to
      & \reali_+ \times  \reali^{n+m-1}
      \\
      & x
      & \mapsto
      & \left(T (t,x),
        \left(X_j (T (t,x), t, x)\right)_{j \ne i}
        \right)
    \end{array}
  \end{equation}
  is a local diffeomorphism. The derivatives of the function $T$ are
  given by
  \begin{align}
    \label{eq:18-new}
    \partial_t T (t,x)
    & =
      - \dfrac{\partial_{2} X_i (T (t,x);t,x)}%
      {v_i \left(T (t,x),X (T (t,x);t,x)\right)}
    \\
    \label{eq:19-new}
    \partial_{x_\ell} T (t,x)
    & =
      -\dfrac{\partial_{3_\ell} X_i \left(T (t,x); t, x\right)}%
      {v_i\left(T (t,x), X\left(T (t,x); t, x\right)\right)}
      \qquad \ell = 1, \ldots, n+m \,.
  \end{align}
  Finally the absolute value of the determinant of the Jacobian matrix
  $D M_i$ at $x$ is
  \begin{equation}
    \label{eq:det-jabobian}
    \frac{1}{v_i\left(T (t,x), X (T (t,x); t,x\right)}
    \exp \int_{t}^{T(t, x)} \sum_{j = 1}^{m+n} \partial_{x_j}
    v_j\left(s, X\left(s; t, x\right)
    \right) \dd s.
  \end{equation}
\end{lemma}

\begin{proof}
  By~\ref{ip:(v)}, the usual Cauchy Theorem for systems of ordinary
  differential equations ensures that, for all
  $(t_o, x_o) \in \reali_+ \times \mathcal{X}$, the Cauchy
  Problem~\eqref{eq:3} admits a unique solution defined on a maximal
  interval $[T_{(t_o,x_o)}, +\infty [$, with
  $T_{(t_o,x_o)} \in [0,t_o]$. Then, the map $T$ defined
  in~\eqref{eq:11} can be written $T (t,x) = T_{(t,x)}$ whenever
  $T_{(t,x)} >0$ and $T(t, x) = 0$ otherwise.  Hence, the
  map~\eqref{eq:16-new} is well defined.

  Once $x \in X (t; \mathopen[0,t\mathclose[, \partial\mathcal{X})$,
  it is clear that there exists at least one index $i$ such
  that~\eqref{eq:15-new} holds. The uniqueness follows, since
  $X (t;\cdot, \cdot)$ is a diffeomorphism.

  Fix $t >0$, $i \in \{1, \ldots, m\}$, and $x \in {\mathbb X}_i^t$.
  Locally around $(t, x)$, the constraint~\eqref{eq:15-new} remains
  valid.  To compute the derivatives of the map $(t,x) \to T (t,x)$,
  differentiating~\eqref{eq:15-new} with respect to~$t$ yields
  \begin{equation*}
    \partial_1 X_i\left(T (t,x); t, x\right) \, \partial_t T (t,x)
    + \partial_{2} X_i\left(T (t,x);t,x\right)
    =
    0
  \end{equation*}
  and so, using~\eqref{eq:ptX},
  \begin{equation*}
    v_i \left(T (t,x), X\left(T (t,x); t, x\right)\right)
    \partial_t T (t,x)
    + \partial_{2} X_i\left(T (t,x);t,x\right)
    =
    0
  \end{equation*}
  which proves~\eqref{eq:18-new}, while a differentiation with respect
  to~$x_\ell$ ($\ell \in \left\{1, \ldots, m+n\right\}$) yields
  \begin{equation*}
    \partial_1 X_i\left(T (t,x);t,x\right) \; \partial_{x_\ell} T (t,x)
    +
    \partial_{3_\ell} X_i\left(T (t,x); t,x\right) = 0
  \end{equation*}
  and so, using~\eqref{eq:ptX},
  \begin{equation*}
    v_i \left(T (t,x), X\left(T(t, x); t, x\right)\right)
    \; \partial_{x_\ell} T (t,x)
    +
    \partial_{3_\ell} X_i\left(T (t,x); t,x\right) = 0,
  \end{equation*}
  which proves~\eqref{eq:19-new}.

  Consider the $\left(n+m\right) \times \left(n+m\right)$ Jacobian
  matrix $D M_i$.  By~\eqref{eq:19-new}, the first row is
  \begin{equation*}
    \left(\partial_{x_1} T(t, x), \cdots,
      \partial_{x_{n+m}} T(t, x)\right)
    =
    \left(- \frac{\partial_{3_1} X_i}{v_i}, \cdots,
      - \frac{\partial_{3_{n+m}} X_i}{v_i}\right),
  \end{equation*}
  where, for simplicity, we omitted the arguments of the functions
  $X_i$ and $v_i$.  The remaining rows, indexed by
  $j \in \left\{1, \ldots, n+m\right\}$, $j \ne i$, of $D M_i$ are
  given by
  \begin{eqnarray*}
    &
    &
      \left(\partial_{x_1} X_j(T(t, x); t, x), \cdots,
      \partial_{x_{n+m}} X_j(T(t, x); t, x)\right)
    \\
    & =
    & \left(-v_j \frac{\partial_{3_1} X_i}{v_i} + \partial_{3_1} X_j,
      \cdots, -v_j \frac{\partial_{3_{n+m}} X_i}{v_i} + \partial_{3_{n+m}} X_j
      \right).
  \end{eqnarray*}
  We compute the determinant of $DM_i$ using Gauss method.  We modify
  all the rows, except the first one, by adding to each row a multiple
  of the first one. In this way the determinant of $DM_i$ equals the
  determinant of the matrix
  \begin{equation*}
    \left(
      \begin{array}{cccc}
        - \frac{\partial_{3_1} X_i}{v_i}
        &
          - \frac{\partial_{3_2} X_i}{v_i}
        & \cdots
        & - \frac{\partial_{3_{n+m}} X_i}{v_i}
          \vspace{.2cm}\\
        \partial_{3_1} X_1
        & \partial_{3_2} X_1
        & \cdots
        & \partial_{3_{n+m}} X_1
          \vspace{.2cm}\\
        \vdots
        & \vdots
        & \vdots
        & \vdots
          \vspace{.2cm}\\
        \partial_{3_1} X_{n+m}
        & \partial_{3_2} X_{n+m}
        & \cdots
        & \partial_{3_{n+m}} X_{n+m}
      \end{array}
    \right)
  \end{equation*}
  in the case $i \ne 1, n+m$, the other cases being entirely similar.
  Therefore
  $\modulo{\det \left(DM_i\right)} = \frac{1}{v_i} \, \modulo{\det
    \left(D_3 X\right)}$.  Using~\eqref{eq:21} and Liouville
  Theorem~\cite[Theorem~1.2, Chapter IV]{MR1929104}, we deduce
  \begin{displaymath}
    \begin{split}
      \modulo{\det \left(DM_i(x)\right)} & = \frac{1}{v_i\left(T
          (t,x); X(T (t,x); t,x)\right)} \exp \int_{t}^{T(t, x)} \tr
      \left(D_x v\left(s, X\left(s; t, x\right) \right) \right)\dd s
      \\
      & = \frac{1}{v_i\left(T (t,x); X(T (t,x); t,x)\right)} \exp
      \int_{t}^{T(t, x)} \sum_{j = 1}^{m+n} \partial_{x_j} v_j\left(s,
        X\left(s; t, x\right) \right) \dd s
    \end{split}
  \end{displaymath}
  which proves~\eqref{eq:det-jabobian}.
\end{proof}

The next two lemmas provide the basic \emph{a priori} and stability
estimates on~\eqref{eq:2}.

\begin{lemma}
  \label{lem:L1inf}
  Let~\ref{ip:(v)} with $k=1$ hold, let
  $p \in \LL\infty(I \times \mathcal{X}; \reali)$,
  $q \in \LL1(I \times \mathcal{X}; \reali)$,
  $u_b \in \LL1 (I \times \partial \mathcal{X}; \reali)$ and
  $u_o \in \LL1 (\mathcal{X};\reali)$. Then, for every $t \in I$ the
  solution to problem~\eqref{eq:2} defined through
  formula~\eqref{eq:8} satisfies the following \emph{a priori}
  estimates:
  \begin{equation}
    \label{eq:17}
    \begin{array}{rcl}
      \norma{u(t)}_{\LL1 (\mathcal{X};\reali)}
      & \leq
      & \displaystyle
        \left(
        \norma{q}_{\LL1([0, t] \times \mathcal{X}; \reali)}
        + \norma{u_o}_{\LL1(\mathcal{X})}
        \right)
        e^{\norma{p}_{\LL\infty ([0,t]\times\mathcal{X};\reali)}t}
      \\
      &
      & \displaystyle
        + \left(
        \sum_{i=1}^{m}
        \iint_{\Gamma_i}
        \modulo{u_b(\tau, \xi)} \,
        v_i (\tau, \xi) \dd \tau \dd \xi
        \right)
        e^{\norma{p}_{\LL\infty ([0,t]\times\mathcal{X};\reali)}t} ,
    \end{array}
  \end{equation}
  where $\Gamma_i = M_i (\mathbb X_i^t)$ with $M_i$ as
  in~\eqref{eq:20-new} and $\mathbb X^i_t$ is as in
  Lemma~\ref{lem:t-new}.  If moreover
  $q \in \LL1\left(I; \LL\infty\left(\mathcal{X};
      \reali\right)\right)$,
  $u_o \in \LL\infty\left(\mathcal X; \reali\right)$, and
  $u_b \in \LL\infty (I \times \partial \mathcal{X}; \reali)$, then
  \begin{equation}
    \label{eq:5}
    \begin{array}{rcl}
      \norma{u(t)}_{\LL\infty (\mathcal{X}; \reali)}
      & \leq
      & \displaystyle
        \left(
        \norma{u_o}_{\LL\infty (\mathcal{X};\reali)}
        +
        \norma{u_b}_{\LL\infty ([0,t] \times \partial \mathcal{X};\reali))}
        +
        \norma{q}_{\LL1 ([0,t]; \LL\infty (\mathcal{X};\reali))}
        \right)
      \\
      &
      & \displaystyle
        \times \exp\left(
        \int_0^t \left(
        \norma{p (\tau)}_{\LL\infty (\mathcal{X};\reali)}
        +
        \norma{\diver_x v (\tau)}_{\LL\infty (\mathcal{X};\reali)}
        \right)\dd\tau
        \right) \,.
    \end{array}
  \end{equation}
\end{lemma}

\begin{proof}
  The proof of the $\LL\infty$ bound directly follows from
  \begin{displaymath}
    \mathcal{E} (\tau, t, x)
    \leq
    \exp \left(
      \norma{p}_{\LL1 ([\tau,t]; \LL\infty (\mathcal{X};\reali))}+
      \norma{\diver_x v}_{\LL1 ([\tau,t]; \LL\infty (\mathcal{X};\reali))}
    \right)\,,
  \end{displaymath}
  and~\eqref{eq:8}. In order to get the $\LL1$ bound, observe that
  $\norma{u (t)}_{\LL1 (\mathcal{X}; \reali)} = \norma{u (t)}_{\LL1 (X
    (t; 0, \mathcal{X}); \reali)} + \norma{u (t)}_{\LL1 (X (t; [0,t[,
    \partial \mathcal{X}); \reali)}$.  We thus consider two cases and
  apply a suitable change of variable.

  By~\eqref{eq:8}, for $t \in I$, we have that
  \begin{equation}
    \label{eq:L1-initial-case}
    \begin{split}
      \int_{X\left(t; 0, \mathcal{X}\right)} \modulo{u(t, x)} \dd x &
      \leq \int_{X\left(t; 0, \mathcal{X}\right)} \modulo{u_o
        \left(X(0; t, x)\right)} \, \mathcal E\left(0, t, x\right) \dd
      x
      \\
      & \quad + \int_{X\left(t; 0, \mathcal{X}\right)} \int_0^t
      \modulo{q\left(\tau, X\left(\tau; t, x\right)\right)} \,
      \mathcal E\left(\tau, t, x\right) \dd \tau \dd x.
    \end{split}
  \end{equation}
  Consider the first term in the right hand side
  of~\eqref{eq:L1-initial-case}.  Using Liouville
  Theorem~\cite[Theorem~1.2, Chapter IV]{MR1929104}, the change of
  variables $\xi = X(0; t, x)$ and the assumptions on $p$,
  \begin{align*}
    \int_{X\left(t; 0, \mathcal{X}\right)}
    \modulo{u_o \left(X(0; t, x)\right)}
    \mathcal E\left(0, t, x\right) \dd x
    & = \int_{\mathcal{X}} \modulo{u_o(\xi)} \exp \left(\int_0^t
      p \left(s, X\left(s; 0, \xi\right)\right) \dd s\right) \dd \xi
    \\
    & \leq
      \norma{u_o}_{\LL1(\mathcal{X})} \,
      e^{\norma{p}_{\LL\infty ([0,t]\times\mathcal{X};\reali)} t}.
  \end{align*}
  Consider the second term in the right hand side
  of~\eqref{eq:L1-initial-case}.  Using the change of variable
  $\xi = X\left(\tau; t, x\right)$,
  \begin{align*}
    & \int_{X\left(t; 0, \mathcal{X}\right)} \int_0^t
      \modulo{q\left(\tau, X\left(\tau; t, x\right)\right)}
      \mathcal E\left(\tau, t, x\right) \dd \tau \dd x
    \\
    =
    &\int_0^t \int_{X\left(\tau; 0, \mathcal{X}\right)}
      \modulo{q(\tau, \xi)}
      \exp\left(\int_\tau^t p \left(s, X(s; \tau, \xi)\right) \dd s
      \right) \dd \xi \dd \tau
    \\
    \leq
    &  \norma{q}_{\LL1(X\left([0, t]; 0, \mathcal{X}\right); \reali)}
      e^{\norma{p}_{\LL\infty ([0,t]\times\mathcal{X};\reali)} t} .
  \end{align*}
  Therefore, using~\eqref{eq:L1-initial-case}, for $t \in I$, we
  deduce
  \begin{equation}
    \label{eq:L1-initial-cond-final}
    \int_{X(t; 0, \mathcal{X})}
    \modulo{u(t, x)} \dd x
    \leq
    \left(\norma{u_o}_{\LL1(\mathcal{X})}
      + \norma{q}_{\LL1(X\left([0, t]; 0, \mathcal{X}\right); \reali)}
    \right)
    e^{\norma{p}_{\LL\infty ([0,t]\times\mathcal{X};\reali)} t}.
  \end{equation}

  To estimate now the term depending on the boundary conditions, for
  $t \in I$, use~\eqref{eq:8}:
  \begin{eqnarray}
    \nonumber
    \int_{X(t; [0, t[, \partial \mathcal{X})}
    \modulo{u(t,x)} \dd x
    & =
    & \sum_{i=1}^{m} \int_{\mathbb X_i^t}
      \modulo{u(t, x)} \dd x
    \\
    \nonumber
    & \leq
    & \sum_{i=1}^{m} \int_{\mathbb X_i^t} \modulo{u_b
      \left(T(t, x), X\left(T(t, x); t, x\right)\right)} \,
      \mathcal{E}
      \left(T(t, x), t, x\right) \dd x
    \\
    \label{eq:L1-boundary-case}
    &
    & + \sum_{i=1}^{m} \int_{\mathbb X_i^t} \int_{T(t,
      x)}^t \modulo{q\left(\tau, X\left(\tau; t, x\right)\right)} \,
      \mathcal E\left(\tau, t, x\right) \dd \tau \dd x.
  \end{eqnarray}
  For $i \in \{1, \ldots, m\}$, use the diffeomorphism $M_i$
  in~\eqref{eq:20-new} as change of variables, i.e.,~$\tau = T(t, x)$,
  $\xi = X\left(T(t, x); t, x\right)$ and we set
  $\Gamma_i = M_i (\mathbb X_i^t)$. Thus, we have
  \begin{align*}
    & \int_{\mathbb X_i^t}
      \modulo{u_b \left(T(t, x), X\left(T(t, x); t, x\right)\right)} \,
      \mathcal E\left(T(t, x), t, x\right) \dd x
    \\
    =
    & \iint_{\Gamma_i}
      \modulo{u_b(\tau, \xi)} \, \exp \left(\int_\tau^t
      p\left(s, X(s; \tau, \xi)\right) \dd s\right) \,
      v_i (\tau, \xi)  \, \dd \tau \, \dd \xi
    \\
    \le
    & e^{\norma{p}_{\LL\infty ([0,t]\times\mathcal{X}; \reali)} t}
      \iint_{\Gamma_i}
      \modulo{u_b(\tau, \xi)} \,
      v_i (\tau, \xi) \, \dd \tau \, \dd \xi.
  \end{align*}
  For $i \in \{1, \ldots, m\}$, using again the change of variables
  $\xi = X \left(\tau; t, x\right)$, define
  \begin{equation}
    \label{eq:27}
    \Xi^i_t
    =
    \left\{
      (\tau,\xi) \in \reali^{1+m+n} \colon
      \tau \in [t, T (t,x)]
      \,,\;
      x \in \mathbb X_t^i
      \,,\;
      \xi = X (\tau;t,x)
    \right\}
  \end{equation}
  and we have
  \begin{eqnarray*}
    &
    & \int_{\mathbb X^i_t} \int_{T (t,x)}^t
      \modulo{q \left(\tau, X\left(\tau; t, x\right)\right)} \,
      \mathcal{E}\left(\tau, t, x\right) \dd{\tau} \dd x
    \\
    & =
    & \iint_{\Xi^i_t}
      \modulo{q(\tau, \xi)}
      \exp \left(\int_\tau^t p \left(s, X(s; \tau, \xi)\right) \dd s\right)
      \dd \tau \dd \xi
    \\
    & \le
    & \norma{q}_{\LL1(\Xi^i_t; \reali)} \,
      e^{\norma{p}_{\LL\infty ([0,t]\times\mathcal{X};\reali)} t}.
  \end{eqnarray*}
  Therefore, using~\eqref{eq:L1-boundary-case}, for $t \in I$, we
  deduce
  \begin{displaymath}
    \int_{X(t; [0, t[, \partial\mathcal{X})}
    \modulo{u(t, x)} \dd x
    \leq
    e^{\norma{p}_{\LL\infty ([0,t]\times\mathcal{X};\reali)} t}
    \sum_{i=1}^{m} \left[ \iint_{\Gamma_i}
      \modulo{u_b(\tau, \xi)}
      v_i (\tau, \xi) \dd \tau \dd \xi
      + \norma{q}_{\LL1(\Xi^i_t; \reali)}\right] \,.
  \end{displaymath}
  This concludes the proof.
\end{proof}

\begin{lemma}
  \label{lem:stability-linear-system}
  Fix $v$ satisfying~\ref{ip:(v)} with $k=1$. Let
  $p_1,p_2 \in \LL\infty(I \times \mathcal{X}; \reali)$,
  $q_1, q_2 \in \LL1(I \times \mathcal{X}; \reali)$ with $u_{b,1}$ and
  $u_{b,2}$ as in \Cref{lem:L1inf} and let $u_{o,1}, u_{o, 2}$
  satisfy~\ref{ip:(u0)}. Define $u_1$ and $u_2$ respectively the
  solutions to
  \begin{equation*}
    \left\{
      \begin{array}{l}
        \partial_t u_1
        +
        \diver_x \left(v \, u_1\right)
        =
        p_1 \, u_1 + q_1
        \\
        u_1 (t,\xi) = u_{b,1} (t,\xi)
        \\
        u_1 (0,x) = u_{o,1} (x)
      \end{array}
    \right.
    \quad \mbox{ and } \quad
    \left\{
      \begin{array}{l}
        \partial_t u_2
        +
        \diver_x \left(v \, u_2\right)
        = p_2 \, u_2 + q_2
        \\
        u_2 (t,\xi) = u_{b,2} (t,\xi)
        \\
        u_2 (0,x) = u_{o,2} (x).
      \end{array}
    \right.
  \end{equation*}
  Then, for every $t \in I$, the following stability estimate holds
  \begin{eqnarray}
    \nonumber
    &
    & \norma{u_1(t) - u_2(t)}_{\LL1 (\mathcal{X};\reali)}
    \\
    \nonumber
    &\le
    & \mathcal{P} (t) \,
      \norma{u_{o,1} - u_{o,2}}_{\LL1 (\mathcal{X}; \reali)} \,
    \\
    \nonumber
    &
    & + \mathcal{P} (t) \,
      \norma{v}_{\LL\infty ([0,t]\times\mathcal{X}; \reali^{n+m})} \,
      \norma{u_{b,1}- u_{b,2}}_{\LL1 (I\times  \partial\mathcal{X};\reali)}
    \\
    \nonumber
    &
    & + \mathcal{P} (t) \,
      \norma{q_1-q_2}_{\LL1 ([0,t]\times\mathcal{X};\reali)}
    \\
    \nonumber
    &
    & + \mathcal{P} (t)
      \left(
      \norma{u_{o,1}}_{\LL1 (\mathcal{X}; \reali)}
      {+}
      \norma{v}_{\LL\infty ([0,t]\times\mathcal{X}; \reali^{n+m})}
      \norma{u_{b,2}}_{\LL1 ([0,t]\times\partial\mathcal{X};\reali)}
      \right)
          \norma{p_1{-}p_2}_{\LL1 ([0,t]; \LL\infty(\mathcal{X}; \reali))}
    \\
    &
    & +
      \mathcal{P} (t) \,
      \norma{q_2}_{\LL1 ([0,t] \times \mathcal{X}; \reali)} \,
      \norma{p_1-p_2}_{\LL1 ([0,t]; \LL\infty(\mathcal{X};\reali))} \,
      \,,
      \label{eq:stability-linear}
  \end{eqnarray}
  where
  $\mathcal{P} (t) = \exp\left(t\, \max\left\{ \norma{p_1}_{\LL\infty
        ([0,t]\times\mathcal{X};\reali)}, \norma{p_2}_{\LL\infty
        ([0,t]\times\mathcal{X};\reali)}\right\}\right)$.
\end{lemma}

\begin{proof}
  Consider $u_1$ and $u_2$ the solutions to the two systems and fix
  $t \in I$.  Define for $i=1, 2$
  \begin{equation*}
    \mathcal E_i\left(\tau, t, x\right)
    = \exp\left(
      \int_\tau^t
      \left(
        p_i\left(s, X (s;t,x)\right)
        -
        \diver_x v \left(s,X (s;t,x)\right)
      \right)
      \dd{s}
    \right).
  \end{equation*}
  We have the decomposition
  \begin{equation}
    \label{eq:stab-decomposition}
    \norma{u_1(t) - u_2(t)}_{\LL1 (\mathcal{X}; \reali)}
    =
    \int_{X\left(t; 0, \mathcal{X}\right)} \modulo{u_1(t) - u_2(t)} \dd x
    + \int_{X(t; [0, t[, \partial \mathcal{X})}
    \modulo{u_1(t) - u_2(t)} \dd x \,.
  \end{equation}
  We treat the two terms in the right hand side
  of~\eqref{eq:stab-decomposition} separately. The first one is dealt
  with the explicit formula~\eqref{eq:8}:
  \begin{align*}
    &
      \quad \int_{X\left(t; 0, \mathcal{X}\right)} \modulo{u_1(t) - u_2(t)} \dd x
    \\
    &
      \le \int_{X\left(t; 0, \mathcal{X}\right)}
      \modulo{u_{o,1}\left(X\left(0; t, x\right)\right)
      \mathcal E_1\left(0, t, x\right)
      - u_{o,2}\left(X\left(0; t, x\right)\right)
      \mathcal E_2\left(0, t, x\right)} \dd x
    \\
    & \quad + \int_{X\left(t; 0, \mathcal{X}\right)} \int_0^t
      \modulo{q_1\left(\tau, X\left(\tau; t, x\right)\right)
      \mathcal E_1\left(\tau, t, x\right)
      - q_2\left(\tau, X\left(\tau; t, x\right)\right)
      \mathcal E_2\left(\tau, t, x\right)} \dd \tau \dd x
    \\
    &
      \le \int_{X\left(t; 0, \mathcal{X}\right)}
      \mathcal E_1\left(0, t, x\right)
      \modulo{u_{o,1}\left(X\left(0; t, x\right)\right)
      - u_{o,2}\left(X\left(0; t, x\right)\right)
      } \dd x
    \\
    &
      \quad + \int_{X\left(t; 0, \mathcal{X}\right)}
      \modulo{u_{o,2}\left(X\left(0; t, x\right)\right)}
      \modulo{
      \mathcal E_1\left(0, t, x\right)
      - \mathcal E_2\left(0, t, x\right)} \dd x
    \\
    & \quad + \int_{X\left(t; 0, \mathcal{X}\right)} \int_0^t
      \mathcal E_1\left(\tau, t, x\right)
      \modulo{q_1\left(\tau, X\left(\tau; t, x\right)\right)
      - q_2\left(\tau, X\left(\tau; t, x\right)\right)
      } \dd \tau \dd x
    \\
    & \quad + \int_{X\left(t; 0, \mathcal{X}\right)} \int_0^t
      \modulo{q_2\left(\tau, X\left(\tau; t, x\right)\right)}
      \modulo{
      \mathcal E_1\left(\tau, t, x\right)
      -
      \mathcal E_2\left(\tau, t, x\right)} \dd \tau \dd x.
  \end{align*}
  Using the two changes of variable $\xi = X\left(0; t, x\right)$ and
  $\xi = X\left(\tau; t, x\right)$, we obtain that
  \begin{align*}
    &
      \quad \int_{X\left(t; 0, \mathcal{X}\right)} \modulo{u_1(t) - u_2(t)} \dd x
    \\
    &
      \le \int_{\mathcal{X}}
      \exp\left(
      \int_0^t
      p_1\left(s, X (s; 0, \xi)\right)
      \dd{s}
      \right)
      \modulo{u_{o,1}\left(\xi\right)
      - u_{o,2}\left(\xi\right)
      } \dd \xi
    \\
    &
      \quad + \int_{\mathcal{X}}
      \modulo{u_{o,2}\left(\xi\right)}
      \modulo{
      \exp\left(
      \int_0^t
      p_1\left(s, X (s; 0, \xi)\right)
      \dd{s}
      \right)
      -
      \exp\left(
      \int_0^t
      p_2\left(s, X (s; 0, \xi)\right)
      \dd{s}
      \right)
      } \dd \xi
    \\
    & \quad + \int_0^t \int_{X\left(\tau; 0, \mathcal{X}\right)}
      \modulo{q_1\left(\tau, \xi\right)
      - q_2\left(\tau, \xi\right)
      }
      \exp\left(
      \int_\tau^t
      p_1\left(s, X (s; \tau, \xi)\right)
      \dd{s}
      \right)
      \dd \xi \, \dd \tau
    \\
    & \quad
      +  \int_0^t \int_{X\left(\tau; 0, \mathcal{X}\right)}
      \modulo{q_2\left(\tau, \xi\right)}
    \\
    & \qquad\qquad \times
      \modulo{
      \exp\left(
      \int_\tau^t
      p_1\left(s, X (s; \tau, \xi)\right)
      \dd{s}
      \right)
      -
      \exp\left(
      \int_\tau^t
      p_2\left(s, X (s; \tau, \xi)\right)
      \dd{s}
      \right)
      } \dd \xi \dd \tau
    \\
    & \leq
      \mathcal{P} (t)
      \left(
      \norma{u_{o,1} - u_{o,2}}_{\LL1 (\mathcal{X}; \reali)}
      +
      \norma{q_1-q_2}_{\LL1 \left(X\left([0,t]; 0, \mathcal{X}\right);
      \reali\right)}
      \right)
    \\
    & \qquad +
      \mathcal{P}(t)
      \norma{u_{o,2}}_{\LL1 (\mathcal{X};\reali)}
      \norma{p_1 - p_2}_{\LL1 ([0,t]; \LL\infty(\mathcal{X}; \reali))}
    \\
    & \qquad +
      \mathcal{P}(t)
      \norma{q_2}_{\LL1\left(X\left([0,t]; 0, \mathcal{X}\right);\reali\right)}
      \norma{p_1 - p_2}_{\LL1 ([0,t]; \LL\infty(\mathcal{X}; \reali))} \,,
  \end{align*}
  where we set
  \begin{equation}
    \label{eq:35}
    \mathcal{P} (t) =   \exp
    \left(
      \max
      \left\{
        {\norma{p_1}_{\LL\infty ([0,t]\times\mathcal{X};\reali)} t}\,,\;
        {\norma{p_2}_{\LL\infty ([0,t]\times\mathcal{X};\reali)} t}
      \right\}
    \right) \,.
  \end{equation}
  Pass now to the second term in the right hand side
  of~\eqref{eq:stab-decomposition}, splitting among the different
  faces $\mathbb X_i^t$ for $i\in \{1, \ldots, m\}$ as defined
  in~\eqref{eq:20-new}:
  \begin{displaymath}
    \int_{X(t; [0, t[, \partial \mathcal{X})}
    \modulo{u_1(t) - u_2(t)} \dd x
    =
    \sum_{i=1}^{m}  \int_{\mathbb X_i^t}
    \modulo{u_1(t) - u_2(t)} \dd x.
  \end{displaymath}
  Fix $i\in \{1, \ldots, m\}$, i.e.  consider each term in the sum
  separately:
  \begin{eqnarray*}
    &
    & \int_{\mathbb X_i^t}
      \modulo{u_1(t) - u_2(t)} \dd x
    \\
    & \le
    & \int_{\mathbb X_i^t}
      \left|
      u_{b,1}\left(T (t,x), X\left(T (t,x); t, x\right)\right)
      \mathcal E_1\left(T (t,x), t, x\right)
      \right.
    \\
    &
    & \qquad\qquad
      \left.
      - u_{b,2}\left(T (t,x),  X\left(T (t,x); t, x\right)\right)
      \mathcal E_2\left(T (t,x), t, x\right)
      \right| \dd x
    \\
    &
    & \quad +
      \int_{\mathbb X_i^t} \int_{T (t,x)}^t
      \modulo{q_1\left(\tau, X\left(\tau; t, x\right)\right)
      \mathcal E_1\left(\tau, t, x\right)
      - q_2\left(\tau, X\left(\tau; t, x\right)\right)
      \mathcal E_2\left(\tau, t, x\right)} \dd \tau \dd x
    \\
    & \le
    & \int_{\mathbb X_i^t}
      \mathcal E_1\left(T (t,x), t, x\right)
    \\
    &
    & \qquad \qquad \times
      \modulo{
      u_{b,1}\left(T (t,x),  X\left(T (t,x); t, x\right)\right)
      -
      u_{b,2}\left(T (t,x),  X\left(T (t,x); t, x\right)\right)
      } \dd x
    \\
    &
    & \quad + \int_{\mathbb X_i^t}
      \modulo{u_{b,2}\left(T (t,x),  X\left(T (t,x); t, x\right)\right)}
      \modulo{
      \mathcal E_1\left(T (t,x), t, x\right)
      -
      \mathcal E_2\left(T (t,x), t, x\right)} \dd x
    \\
    &
    & \quad + \int_{\mathbb X_i^t} \int_{T (t,x)}^t
      \mathcal E_1\left(\tau, t, x\right)
      \modulo{q_1\left(\tau, X\left(\tau; t, x\right)\right)
      - q_2\left(\tau, X\left(\tau; t, x\right)\right)
      } \dd \tau \dd x
    \\
    &
    & \quad + \int_{\mathbb X_i^t} \int_{T (t,x)}^t
      \modulo{q_2\left(\tau, X\left(\tau; t, x\right)\right)}
      \, \modulo{
      \mathcal E_1\left(\tau, t, x\right)
      -
      \mathcal E_2\left(\tau, t, x\right)} \dd \tau \dd x.
  \end{eqnarray*}
  We now use the diffeomorphism $M_i$ as defined in~\eqref{eq:20-new},
  for $i \in \{1, \ldots, m\}$, and we use the set $\Xi^i_t$ as
  in~\eqref{eq:27}. We thus obtain, using~\eqref{eq:35}, that
  \begin{align*}
    &
      \quad \int_{\mathbb X_i^t}
      \modulo{u_1(t,x) - u_2(t,x)} \dd x
    \\
    & \leq
      \iint_{\Gamma_i}
      \exp\left(\int_\tau^t p_1 \left(s,X (s;\tau,\xi)\right)\dd{s}\right)
      \modulo{u_{b,1}(\tau,\xi) - u_{b,2}(\tau, \xi)} \,
      v_i (\tau,\xi) \dd \xi \dd\tau
    \\
    &
      \quad +
      \iint_{\Gamma_i} \modulo{u_{b,2}(\tau, \xi)}
    \\
    & \qquad \times
      \modulo{
      \exp\left(\int_\tau^t p_1 \left(s,X (s;\tau,\xi)\right)\dd{s}\right)
      -
      \exp\left(\int_\tau^t p_2 \left(s,X (s;\tau,\xi)\right)\dd{s}\right)
      } \,
      v_i (\tau,\xi) \dd \xi \dd\tau
    \\
    & \quad + \iint_{\Xi^i_t}
      \exp\left(\int_\tau^t p_1 \left(s,X (s;\tau,\xi)\right)\dd{s}\right)
      \modulo{q_1(\tau, \xi)- q_2(\tau, \xi)} \dd \tau \dd \xi
    \\
    & \quad + \iint_{\Xi^i_t}
      \modulo{q_2(\tau, \xi)}
    \\
    & \qquad \times
      \modulo{
      \exp\left(\int_\tau^t p_1 \left(s,X (s;\tau,\xi)\right)\dd{s}\right)
      -
      \exp\left(\int_\tau^t p_2 \left(s,X (s;\tau,\xi)\right)\dd{s}\right)
      }
      \dd \tau \dd \xi
    \\
    & \leq
      \mathcal{P}(t) \,
      \norma{v}_{\LL\infty ([0,t]\times\mathcal{X}; \reali^{n+m})} \,
      \norma{u_{b,1} - u_{b,2}}_{\LL1 (\Gamma_i; \reali)}
    \\
    & \quad +
      \mathcal{P}(t) \, \norma{v}_{\LL\infty ([0,t]\times\mathcal{X}; \reali^{n+m})}
      \,
      \norma{u_{b,2}}_{\LL1 (\Gamma_i; \reali)} \,
      \norma{p_1 - p_2}_{\LL1 ([0,t]; \LL\infty(\mathcal{X}; \reali))}
    \\
    & \quad +
      \mathcal{P}(t) \, \norma{q_1 - q_2}_{\LL1 (\Xi^i_t; \reali)}
    \\
    & \quad +
      \mathcal{P}(t) \, \norma{q_2}_{\LL1 (\Xi^i_t; \reali)} \,
      \norma{p_1 - p_2}_{\LL1 ([0,t]; \LL\infty(\mathcal{X}; \reali))}
    \\
    & \le \mathcal{P}(t) \left(\norma{v}_{\LL\infty ([0,t]\times\mathcal{X};
      \reali^{n+m})} \,
      \norma{u_{b,1} - u_{b,2}}_{\LL1 (\Gamma_i; \reali)}
      + \norma{q_1 - q_2}_{\LL1 (\Xi^i_t; \reali)}\right)
    \\
    & \quad +
      \mathcal{P}(t)
      \norma{v}_{\LL\infty ([0,t]\times\mathcal{X}; \reali^{n+m})} \,
      \norma{u_{b,2}}_{\LL1 (\Gamma_i; \reali)} \,
      \norma{p_1 - p_2}_{\LL1 ([0,t]; \LL\infty(\mathcal{X};\reali))}
    \\
    & \quad +
      \mathcal{P}(t)
      \norma{q_2}_{\LL1 (\Xi^i_t; \reali)} \,
      \norma{p_1 - p_2}_{\LL1 ([0,t]; \LL\infty(\mathcal{X};\reali))} \,.
  \end{align*}
  Therefore, using~\eqref{eq:stab-decomposition}, we deduce that
  \begin{eqnarray*}
    &
    & \norma{u_1(t) - u_2(t)}_{\LL1 (\mathcal{X}; \reali)}
    \\
    & \le
    & \mathcal{P}(t)
      \left(
      \norma{u_{o,1} - u_{o,2}}_{\LL1 (\mathcal{X}; \reali)}
      +
      \norma{q_1-q_2}_{\LL1 (X;([0,t]; 0, \mathcal{X});\reali)}
      \right)
    \\
    &
    & \quad +
      \mathcal{P}(t)
      \norma{u_{o,2}}_{\LL1 (\mathcal{X};\reali)}
      \norma{p_1 - p_2}_{\LL1 ([0,t]; \LL\infty(\mathcal{X}; \reali))}
    \\
    &
    & \quad +
      \mathcal{P}(t)
      \norma{q_2}_{\LL1 (X;([0,t]; 0, \mathcal{X});\reali)} \,
      \norma{p_1 - p_2}_{\LL1 ([0,t]; \LL\infty(\mathcal{X}; \reali))}
    \\
    &
    & \quad + \sum_{i=1}^{m} \,
      \mathcal{P}(t) \left(\norma{v}_{\LL\infty ([0,t]\times\mathcal{X};
      \reali^{n+m})} \,
      \norma{u_{b,1} - u_{b,2}}_{\LL1 (\Gamma_i; \reali)}
      + \norma{q_1 - q_2}_{\LL1 (\Xi^i_t; \reali)}\right)
    \\
    &
    & \quad + \sum_{i=1}^{m} \,
      \mathcal{P}(t)
      \norma{v}_{\LL\infty ([0,t]\times\mathcal{X}; \reali^{n+m})} \,
      \norma{u_{b,2}}_{\LL1 (\Gamma_i;\reali)} \,
      \norma{p_1 - p_2}_{\LL1 ([0,t]; \LL\infty(\mathcal{X}; \reali))}
    \\
    &
    & \quad + \sum_{i=1}^{m} \,
      \mathcal{P}(t)
      \norma{q_2}_{\LL1 (\Xi^i_t; \reali)} \,
      \norma{p_1 - p_2}_{\LL1 ([0,t]; \LL\infty(\mathcal{X};\reali))}
    \\
    & \le
    & \mathcal{P}(t) \norma{u_{o,1} - u_{o,2}}_{\LL1 (\mathcal{X}; \reali)}
    \\
    &
    & \quad + \mathcal{P}(t)
      \norma{v}_{\LL\infty ([0,t]\times\mathcal{X}; \reali^{n+m})}
      \norma{u_{b,1}- u_{b,2}}_{\LL1 ([0,t]\times \partial\mathcal{X};\reali)}
    \\
    &
    & \quad + \mathcal{P}(t)
      \norma{q_1-q_2}_{\LL1 ([0,t]\times\mathcal{X};\reali)}
    \\
    &
    & \quad + \mathcal{P}(t)
      \left(
      \norma{u_o}_{\LL1 (\mathcal{X}; \reali^k)}
      +
      \norma{v}_{\LL\infty ([0,t]\times\mathcal{X}; \reali^{n+m})}
      \norma{u_{b,2}}_{\LL1 ([0,t]\times\partial\mathcal{X};\reali)}\right)
      \norma{p_1-p_2}_{\LL1 ([0,t]; \LL\infty(\mathcal{X}; \reali))}
    \\
    &
    & \quad + \mathcal{P}(t)
      \norma{q_2}_{\LL1 ([0,t]\times \mathcal{X};\reali)} \,
      \norma{p_1-p_2}_{\LL1 ([0,t]; \LL\infty(\mathcal{X};\reali))} \,,
  \end{eqnarray*}
  proving~\eqref{eq:stability-linear}.
\end{proof}

\begin{proposition}
  \label{prop:ClassicalSolution}
  Let $v$ satisfy~\ref{ip:(v)} with $k=1$,
  $p \in \LL\infty(I \times \mathcal{X}; \reali)$,
  $q \in \LL1(I \times \mathcal{X}; \reali)$,
  $u_b \in \LL1 (I \times \partial \mathcal{X}; \reali)$ and $u_o$
  satisfy~\ref{ip:(u0)} with $k=1$. Then, formula~\eqref{eq:8} defines
  a solution $u = u (t,x)$ to~\eqref{eq:2} in the sense of
  Definition~\ref{def:mvsol}. Moreover,
  $u \in \C0 (I; \LL1 (\mathcal{X}; \reali))$.
\end{proposition}

\begin{proof}
  The first part of the proof amounts to a careful piecing together
  various proofs found in the literature. In particular, the part of
  the solution depending on the initial data is dealt with exactly as
  in~\cite[Lemma~2.7]{Elena2015} and~\cite[Lemma~5.1]{Herty2011}. The
  part depending on the boundary datum is treated in the same way,
  exploiting the change of variables detailed in
  Lemma~\ref{lem:t-new}.

  To prove the $\C0$ regularity of the solution with respect to time,
  fix a $\bar t \in I$ and a sequence $t_h$, with $t_h \in I$,
  converging to $\bar t$. Then, assuming first that $t_h > t$, we have
  \begin{eqnarray*}
    \norma{u (t_h) - u (\bar t)}_{\LL1 (\mathcal{X};\reali)}
    & =
    & \int_{X (t_h;0,\mathcal{X})} \modulo{u (t_h, x) - u (\bar t, x)} \dd{ x}
    \\
    &
    & + \int_{\mathcal{X} \setminus (X (t_h;0,\mathcal{X}) \cup X (\bar t;[0,\bar t[,\partial \mathcal{X}))}
      \modulo{u (t_h, x) - u (\bar t, x)} \dd{ x}
    \\
    &
    & + \int_{X (\bar t;[0,\bar t[,\partial \mathcal{X})} \modulo{u (t_h, x) - u (\bar t, x)} \dd{ x} \,.
  \end{eqnarray*}
  The second term vanishes as $h \to +\infty$, since it is the
  integral of a bounded quantity over a set of vanishing
  measure. Consider now the first term, the third one can be treated
  similarly.
  \begin{eqnarray*}
    &
    & \int_{X (t_h;0,\mathcal{X})} \modulo{u (t_h, x) - u (\bar t, x)} \dd{ x}
    \\
    & =
    & \int_{\mathcal{X}} \modulo{u (t_h, x) - u (\bar t, x)} \,
      \caratt{X (t_h;0,\mathcal{X})} ( x) \, \dd{ x}
    \\
    & \leq
    & \int_{\mathcal{X}}
      \modulo{
      u_o\left(X (0; t_h,x)\right)\, \mathcal{E} (\tau, t_h, x)
      -
      u_o\left(X (0; \bar t,x)\right)\, \mathcal{E} (\tau, \bar t, x)}
      \,
      \caratt{X (t_h;0,\mathcal{X})} ( x)
      \dd x
    \\
    & +
    & \int_{\mathcal{X}}
      \left|
      \int_0^{t_h}
      q\left(\tau, X (\tau;t_h, x)\right) \,
      \mathcal{E}\left(\tau, t_h, x\right) \,
      \dd\tau
      \right.
    \\
    &
    & \qquad
      \left.
      -
      \int_0^{\bar t}
      q\left(\tau, X (\tau; \bar t, x)\right) \,
      \mathcal{E}\left(\tau, \bar t, x\right) \,
      \dd \tau
      \right|  \,
      \caratt{X (t_h;0,\mathcal{X})} ( x)\dd x
  \end{eqnarray*}
  As $h \to +\infty$, we have that
  \begin{eqnarray*}
    u_o\left(X (0; t_h,x)\right)\, \mathcal{E} (\tau, t_h, x)
    & \to
    & u_o\left(X (0; \bar t,x)\right)\, \mathcal{E} (\tau, \bar t, x)
    \\
    \int_0^{t_h}
    q\left(\tau, X (\tau;t_h, x)\right) \,
    \mathcal{E}\left(\tau, t_h, x\right) \,
    \dd\tau
    & \to
    & \int_0^{\bar t}
      q\left(\tau, X (\tau; \bar t, x)\right) \,
      \mathcal{E}\left(\tau, \bar t, x\right) \,
      \dd\tau
  \end{eqnarray*}
  for a.e.~$x \in \mathcal{X}$, so that the corresponding integrals
  vanish by Lebesgue Dominated Convergence Theorem, which we can apply
  thanks to the $\LL1$ \emph{a priori} bound \eqref{eq:17}.
\end{proof}

\subsection{The General Case of a System}

Below, in the various estimates we use the following norms:
\begin{displaymath}
  \begin{array}{@{}r@{\,}c@{\,}l@{\quad}r@{\,}c@{\,}l@{}}
    \norma{u}_{\LL1 (\mathcal{X};\reali^k)}
    & =
    & \sum_{h=1}^k \int_{\mathcal{X}} \modulo{u^h (x)} \dd{x}
    &
      \norma{u}_{\LL\infty (I\times\mathcal{X}; \reali^k)}
    & =
    & \sum_{h=1}^k \norma{u^h}_{\LL\infty (I\times\mathcal{X}; \reali)}
    \\
    \norma{u}_{\LL\infty(I;\LL1(\mathcal{X}; \reali^k))}
    & =
    & \sum_{h = 1}^k
      \norma{u^h}_{\LL\infty(I;\LL1\left(\mathcal{X}; \reali\right))}\,.
  \end{array}
\end{displaymath}

\begin{proofof}{Theorem~\ref{thm:main}}
  The proof is divided in several steps. Let $I = [0,T]$ for $T>0$.

  \paragraph{Construction of the Operator $\mathcal T$.}  In the
  Banach space $\C0 (I; \LL1 (\mathcal{X}; \reali^k))$, for
  \begin{equation}
    \label{eq:24}
    M > \norma{u_o}_{\LL1 (\mathcal{X};\reali^k)} + 1,
  \end{equation}
  introduce the closed subset $X$ and the norm $\norma{\cdot}_X$:
  \begin{align}
    \label{eq:X_M}
    X
    = \
    & \left\{
      w \in \C0(I; \LL1(\mathcal{X}; \reali^k)) \colon
      \norma{w}_{\LL\infty(I;\LL1(\mathcal{X}; \reali^k)) }
      \leq
      M
      \right\} \,,
    \\
    \label{eq:X_M-norm}
    \norma{w}_X
    = \
    & \sum_{h = 1}^k
      \norma{w^h}_{\LL\infty(I;\LL1\left(\mathcal{X}; \reali\right))} \,.
  \end{align}
  Define the operator
  \begin{equation}
    \label{eq:operator-T}
    \begin{array}{rccc}
      \mathcal T:
      & X
      & \longrightarrow
      & X
      \\
      &
        w
      & \longmapsto
      & u \equiv \left(u^1, \ldots, u^k\right)
    \end{array}
  \end{equation}
  where, for every $h \in \left\{1, \ldots, k\right\}$, $u^h$ solves
  \begin{equation}
    \label{eq:T-sol}
    \!\!\!\!\left\{
      \begin{array}{@{\,}l@{\qquad}r@{\,}c@{\,}l@{}}
        \partial_t u^h
        +
        \diver_x \left(v^h (t,x) u^h\right)
        =
        p^h \left(t, x, w(t)\right) u^h
        \\
        \hspace{4.6cm} + q^h \left(t, x, w(t, x), w(t)\right)
        & (t,x)
        & \in
        & I \times\mathcal{X}
        \\
        u^h (t,\xi) = u_b^h \left(t, \xi, w (t)\right)
        & (t,\xi)
        & \in
        & I \times \partial\mathcal{X}
        \\
        u^h (0,x) = u_o^h (x)
        & x
        & \in
        & \mathcal{X}\,.
      \end{array}
    \right.
  \end{equation}

  \paragraph{$\mathcal T$ is Well Defined.}  We prove that, for
  $w \in X$ and $h \in \left\{1, \ldots, k\right\}$, the source term
  in~\eqref{eq:T-sol}
  \begin{displaymath}
    \mathcal{G}^h (t,x, u^h)
    =
    \mathcal{P}^h (t, x) \, u^h + \mathcal{Q}^h (t, x)
    \quad \mbox{ where } \quad
    \begin{array}{r@{\,}c@{\,}l}
      \mathcal{P}^h (t, x)
      & =
      & p^h \left(t, x, w(t)\right)
      \\
      \mathcal{Q}^h (t, x)
      & =
      & q^h \left(t, x, w(t, x), w(t)\right)
    \end{array}
  \end{displaymath}
  is such that
  $\mathcal{P}^h \in \LL\infty(I \times \mathcal{X}; \reali)$ and
  $\mathcal{Q}^h \in \LL1(I \times \mathcal{X}; \reali)$.

  By~\ref{hyp:g_a}, for every $t \in I$ and $x \in \mathcal{X}$, using
  also~\eqref{eq:X_M}, we have
  \begin{eqnarray}
    \nonumber
    \modulo{\mathcal{P}^h (t,x)}
    & =
    & \modulo{p^h \left(t, x, w(t)\right)}
      \; \le \, P_1 + P_2 \; \norma{w(t)}_{\LL1(\mathcal{X}; \reali^k)} \,;
    \\
    \label{eq:G_a_Linfty}
    \norma{\mathcal{P}^h}_{\LL\infty(I \times \mathcal{X}; \reali)}
    & \le
    & P_1 + P_2\, M,
  \end{eqnarray}
  proving that $(t,x) \mapsto \mathcal{P}^h(t,x)$ is in
  $\LL\infty (I \times \mathcal{X}; \reali)$. On the other hand,
  by~\ref{hyp:g_b} we have
  \begin{eqnarray}
    \nonumber
    &
    & \norma{\mathcal{Q}^h}_{\LL1 ([0,T]\times\mathcal{X};\reali)}
    \\
    \nonumber
    & =
    & \int_0^T \int_{\mathcal{X}}
      \modulo{\mathcal{Q}^h (t,x)}
      \dd x\, \dd t
    \\
    \nonumber
    & =
    & \int_0^T \int_{\mathcal{X}}
      \modulo{q^h\left(t, x, w(t,x), w(t)\right)}
      \dd x\, \dd t
    \\
    \nonumber
    & \leq
    & Q_1 \int_0^T \int_{\mathcal{X}} \norma{w(t,x)} \dd x\, \dd t
    \\
    \nonumber
    & \quad
    & + \int_0^T \int_{\mathcal{X}} Q_2(x)
      \norma{w(t)}_{\LL1(\mathcal{X};\reali^k)}\dd x\, \dd t
      + Q_3 \int_0^T \int_{\mathcal{X}} \norma{w(t,x)} \,
      \norma{w(t)}_{\LL1(\mathcal{X};\reali^k)}
      \dd x\, \dd t
    \\
    \label{eq:34}
    & \leq
    & Q_1 T \norma{w}_X + \norma{Q_2}_{\LL1(\mathcal{X}; \reali)}
      T \norma{w}_X
      + Q_3 T \norma{w}_X^2,
  \end{eqnarray}
  proving that $(t, x) \mapsto \mathcal{Q}^h (t,x)$ is in
  $\LL1(I \times \mathcal{X}; \reali)$.

  Now we prove that, for every $w \in X$ and
  $h \in \left\{1, \ldots, k\right\}$, the boundary term
  $\mathcal{U}_b^h (t,\xi) = u^h_b \left(t, \xi, w(t)\right)$
  in~\eqref{eq:T-sol} satisfies
  $\mathcal{U}_b^h \in \LL1 (I \times \partial \mathcal{X}; \reali)$.
  By~\ref{ip:(ub)} we have
  \begin{eqnarray*}
    \norma{\mathcal{U}^h_b}_{\LL1 (I\times \partial\mathcal{X}; \reali)}
    & =
    & \int_0^T \!\! \int_{\partial \mathcal{X}}
      \modulo{u_b^h\left(t, \xi, w(t)\right)}
      \dd \xi \dd t
    \\
    & \leq
    & \int_0^T \!\! \int_{\partial \mathcal{X}} B(\xi)
      \norma{w(t)}_{\LL1(\mathcal{X}; \reali^k)}\dd \xi \dd t
      +
      \int_0^T \!\! \int_{\partial \mathcal{X}} B(\xi)
      \dd \xi \dd t
    \\
    & \leq
    & \norma{B}_{\LL1(\partial \mathcal{X}; \reali)}
      \left(\norma{w}_X + 1\right) \, T \,.
  \end{eqnarray*}
  Hence Proposition~\ref{prop:ClassicalSolution} applies
  to~\eqref{eq:T-sol}.  To conclude this step, we need to show that
  the solution
  $u(t, x) \equiv \left(u^1(t, x), \ldots, u^k(t,x)\right)$ belongs to
  $X$ in~\eqref{eq:X_M}.  By~\eqref{eq:17}, \eqref{eq:G_a_Linfty},
  \eqref{eq:34} and since $w \in X$, for $t \in I$,
  \begin{eqnarray*}
    \norma{u^h (t)}_{\LL1 (\mathcal{X};\reali)}
    & \leq
    & e^{(P_1 + P_2 M)t}
      \left(
      \norma{\mathcal{Q}^h}_{\LL1 ([0,t]\times\mathcal{X};\reali)}
      +
      \norma{u_o^h}_{\LL1 (\mathcal{X};\reali)}\right)
    \\
    &
    & + e^{(P_1 + P_2 M)t}
      \sum_{i=1}^{m}
      \iint_{\Gamma_i}
      \modulo{u_b^h\left(\tau,\xi,w (\tau)\right)} \, v^h_i (\tau,\xi)
      \dd\tau \dd\xi
    \\
    & \leq
    &  \Big[
      \left(
      Q_1
      +
      \norma{Q_2}_{\LL1 (\mathcal{X};\reali)}
      +
      Q_3 \, \norma{w}_X
      \right)
      T \,
      \norma{w}_X
      +
      \norma{u_o^h}_{\LL1 (\mathcal{X};\reali)}
    \\
    &
    & +
      \norma{B}_{\LL1(\partial \mathcal{X}; \reali)} \,
      \norma{v}_{\LL\infty (I\times\mathcal{X};\reali^{k\times(n+m)})} \,
      T \left(\norma{w}_X + 1\right) \Big]
      e^{(P_1 + P_2 M)t}
    \\
    & \leq
    &  \Big[
      \left(
      Q_1
      +
      \norma{Q_2}_{\LL1 (\mathcal{X};\reali)}
      +
      Q_3 \, M
      \right)
      T \, M
      +
      \norma{u_o^h}_{\LL1 (\mathcal{X};\reali)}
    \\
    &
    & +
      \norma{B}_{\LL1(\partial \mathcal{X}; \reali)} \,
      \norma{v}_{\LL\infty (I\times\mathcal{X};\reali^{k\times(n+m)})}\,
      T (M + 1) \Big]
      e^{(P_1 + P_2 M)t}
    \\
    & \leq
    & \left(
      \norma{u_o^h}_{\LL1 (\mathcal{X};\reali)}
      +
      \frac{1}{2k}\right)
      e^{(P_1 + P_2 M)T},
  \end{eqnarray*}
  whence $\norma{u (t)}_{\LL1 (\mathcal{X};\reali^k)} \leq M$, once
  $T$ is sufficiently small, thanks to the choice~\eqref{eq:24} of
  $M$.

  \paragraph{$\mathcal T$ is a Contraction.} Fix $\hat w$ and
  $\check w$ in $X_M$ and call $\hat u =\mathcal{T} \hat w$,
  $\check u = \mathcal{T} \check w$. Use the notation
  \begin{displaymath}
    \!\!\!
    \begin{array}{@{}r@{\,}c@{\,}lr@{\,}c@{\,}lr@{\,}c@{\,}l@{}}
      \hat{\mathcal{P}}^h (t,x)
      & =
      & p^h\left(t, x, \hat w (t)\right),
      & \hat{\mathcal{Q}}^h (t,x)
      & =
      & q^h\left(t, x, \hat w (t,x), \hat w (t)\right),
      & \hat{\mathcal{U}}_b^h (t,\xi)
      & =
      & u_b^h\left(t, \xi, \hat w (t)\right),
      \\
      \check{\mathcal{P}}^h (t,x)
      & =
      & p^h\left(t, x, \check w (t)\right),
      & \check{\mathcal{Q}}^h (t,x)
      & =
      & q^h\left(t, x, \check w (t,x), \check w (t)\right),
      & \check{\mathcal{U}}_b^h (t,\xi)
      & =
      & u_b^h\left(t, \xi, \check w (t)\right).
    \end{array}
  \end{displaymath}
  Then, by Lemma~\ref{lem:stability-linear-system} and
  by~\eqref{eq:G_a_Linfty}, we have:
  \begin{align}
    \nonumber
    & \norma{\hat u^h (t) - \check u^h (t)}_{\LL1 (\mathcal{X}; \reali)}
    \\
    \nonumber
    \le
    & \, e^{(P_1 + P_2 M) t} \,
      \norma{v}_{\LL\infty ([0,t]\times\mathcal{X};\reali^{n+m})}
      \norma{\hat{\mathcal{U}}_b^h - \check{\mathcal{U}}_b^h}
      _{\LL1 ([0,t] \times \partial\mathcal{X};\reali)}
    \\
    \nonumber
    & \quad +
      e^{(P_1 + P_2 M) t}
      \norma{\hat{\mathcal{Q}}^h - \check{\mathcal{Q}}^h}
      _{\LL1 ([0,t] \times \mathcal{X};\reali)}
    \\
    \nonumber
    & \quad +
      \left(
      M
      +
      \norma{v}_{\LL\infty (I\times \mathcal{X}; \reali^{k\times (n+m)})}
      \norma{\check{\mathcal{U}}^h_b}_{\LL1 ([0,t]\times\partial\mathcal{X};\reali)}
      +\norma{\check{\mathcal{Q}}^h}_{\LL1 ([0,t]\times\mathcal{X};\reali)}
      \right)
    \\
    \label{eq:28}
    & \qquad
      \times e^{(P_1+P_2 M)t}
      \norma{\hat{\mathcal{P}}^h - \check{\mathcal{P}}^h}_{\LL1 ([0,t]; \LL\infty(\mathcal{X}; \reali))} \,.
  \end{align}
  By~\ref{hyp:g_a} we have:
  \begin{align}
    \nonumber
    \norma{\hat{\mathcal{P}}^h - \check{\mathcal{P}}^h}
    _{\LL1([0, t]; \LL\infty(\mathcal{X}; \reali))}
    & \le \int_0^t \norma{\hat{\mathcal{P}}^h(s) - \check{\mathcal{P}}^h(s)}
      _{\LL\infty(\mathcal{X}; \reali)} \dd s
    \\
    \nonumber
    & \le P_2
      \int_0^t \norma{\hat w(s) - \check w(s)}
      _{\LL1(\mathcal{X}; \reali^k)} \dd s
    \\
    \label{eq:29}
    & \le P_2 \, \norma{\hat w - \check w}_X \, T \,.
  \end{align}
  By~\ref{hyp:g_b} we have:
  \begin{eqnarray}
    \nonumber
    &
    & \norma{\hat{\mathcal{Q}}^h - \check{\mathcal{Q}}^h}_{\LL1 ([0,t] \times \mathcal{X};\reali)}
    \\
    \nonumber
    & \leq
    & Q_1 \, \norma{\hat w - \check w}_{\LL1 ([0,t]\times\mathcal{X}; \reali^k)}
      +
      Q_3 \, \norma{\hat w}_{\LL\infty([0,t];\LL1 (\mathcal{X}; \reali^k))}
      \, \norma{\hat w - \check w}_{\LL1 ([0,t]\times\mathcal{X}; \reali^k)}
    \\
    \nonumber
    &
    & +
      Q_3 \, \norma{\check w}_{\LL\infty ([0,t]; \LL1 (\mathcal{X}; \reali^k))} \,
      \, \norma{\hat w - \check w}_{\LL1([0,t] \times \mathcal{X}; \reali^k)}
    \\
    \label{eq:30}
    & \leq
    & (Q_1 + 2 \, M \, Q_3) \, \norma{\hat w - \check w}_X \, T \,.
  \end{eqnarray}
  Similarly, by~\ref{ip:(ub)}, we have:
  \begin{equation}
    \label{eq:31}
    \norma{\hat{\mathcal{U}}_b^h - \check{\mathcal{U}}_b^h}_{\LL1 ([0,t] \times \partial\mathcal{X};\reali)}
    \leq
    \norma{B}_{\LL1 (\partial \mathcal{X}; \reali)} \,
    \norma{\hat w - \check w}_X \, T \,.
  \end{equation}
  Therefore $\mathcal{T}$ is a contraction as soon as $T$ is
  sufficiently small.

  \paragraph{Existence of a Solution for Small Times.} Proving that
  the unique fixed point of $\mathcal{T}$ solves~\eqref{eq:1} in the
  sense of Definition~\ref{def:sol} amounts to pass to the limit in
  the integral inequality~\eqref{eq:mv}. This is possible thanks to
  the strong convergence ensured by the choice~\eqref{eq:X_M-norm} of
  the norm in $X$. The proof of~\ref{thm:it:1} is completed.

  \paragraph{Uniqueness.} Assume that~\eqref{eq:32} admits the
  solutions $\hat u$ and $\check u$ in the sense of
  Definition~\ref{def:sol}. Then, their difference
  $\delta = \hat u - \check u$ solves
  \begin{displaymath}
    \left\{
      \begin{array}{l}
        \partial_t \delta^h + \div \left(v^h (t,x) \, \delta^h\right)
        =
        \hat{\mathcal{G}}^h (t,x)  - \check{\mathcal{G}}^h (t,x)
        \\
        \delta^h (t,\xi)
        =
        \hat{\mathcal{U}}^h_b (t,\xi) - \check{\mathcal{U}}^h_b (t,\xi)
        \\
        \delta^h (0,x) = 0
      \end{array}
    \right.
  \end{displaymath}
  in the sense of Definition~\ref{def:sol}, where
  \begin{displaymath}
    \begin{array}{r@{\,}c@{\,}l@{;\qquad}r@{\,}c@{\,}l}
      \hat{\mathcal{G}}^h (t,x)
      & =
      & p^h\left(t,x,\hat u (t)\right) \hat u^h
        + q^h\left(t,x,\hat u, \hat  u (t)\right)
      & \hat{\mathcal{U}}^h_b (t,\xi)
      & =
      & \hat u_b^h \left(t,\xi,\hat u (t)\right) \,;
      \\
      \check{\mathcal{G}}^h (t,x)
      & =
      & p^h\left(t,x,\check u (t)\right) \check u^h
        + q^h\left(t,x,\check u, \check  u (t)\right)
      & \check{\mathcal{U}}^h_b (t,\xi)
      & =
      & \check u_b^h \left(t,\xi,\check u (t)\right) \,.
    \end{array}
  \end{displaymath}
  A straightforward application of the classical doubling of variable
  method~\cite{Kruzkov}, see~\cite[Lemma~16, Lemma~17]{Martin},
  \cite[Theorem~7.28]{MalekEtAlBook}, and
  also~\cite[Proposition~2.8]{Elena2015}, leads to the stability
  estimate
  \begin{eqnarray*}
    \norma{\delta^h (t)}_{\LL1 (\mathcal{X}; \reali)}
    & \leq
    & \int_0^t
      \norma{
      \hat{\mathcal{G}}^h (\tau)  - \check{\mathcal{G}}^h (\tau)
      }_{\LL1 (\mathcal{X}; \reali)} \dd\tau
    \\
    &
    & +
      \norma{v^h}_{\LL\infty (I\times\mathcal{X};\reali^{n+m})} \;
      \int_0^t
      \norma{
      \hat{\mathcal{U}}^h_b (\tau) - \check{\mathcal{U}}^h_b (\tau)
      }_{\LL1 (\partial \mathcal{X}; \reali)} \dd\tau \,.
  \end{eqnarray*}
  The assumptions~\ref{hyp:g_a} and~\ref{hyp:g_b} allow now to use
  Gronwall Lemma, proving that $\delta \equiv 0$.

  \bigskip

  \paragraph{Continuous Dependence on the Initial Datum.} With the
  notation in~\ref{thm:it:ID}, define
  \begin{displaymath}
    \begin{array}{@{}r@{\,}c@{\,}lr@{\,}c@{\,}lr@{\,}c@{\,}l@{}}
      \hat{\mathcal{P}}^h (t,x)
      & =
      & p^h\left(t, x, \hat u (t)\right),
      & \hat{\mathcal{Q}}^h (t,x)
      & =
      & q^h\left(t, x, \hat u (t,x), \hat u (t)\right),
      & \hat{\mathcal{U}}_b^h (t,\xi)
      & =
      & u_b^h\left(t, \xi, \hat u (t)\right),
      \\
      \check{\mathcal{P}}^h (t,x)
      & =
      & p^h\left(t, x, \check u (t)\right),
      & \check{\mathcal{Q}}^h (t,x)
      & =
      & q^h\left(t, x, \check u (t,x), \check u (t)\right),
      & \check{\mathcal{U}}_b^h (t,\xi)
      & =
      & u_b^h\left(t, \xi, \check u (t)\right),
    \end{array}
  \end{displaymath}
  for $t \in I$ and $h \in \left\{1, \ldots, k\right\}$. A further
  application of Lemma~\ref{lem:stability-linear-system} allows to
  estimate the difference between the solutions $\hat u$ and
  $\check u$.
  \begin{equation}
    \begin{split}
      \label{eq:stab-est-proof}
      & \norma{\hat{u}^h(t) - \check{u}^h(t)}_{\LL1
        (\mathcal{X};\reali)}
      \\
      \le & \; e^{(P_1 + P_2 M) t} \left(\norma{\hat u_{o,h} - \check
          u_{o,h}}_{\LL1 (\mathcal{X}; \reali)} + \norma{v}_{\LL\infty
          ([0,t]\times\mathcal{X};\reali^{n+m})}
        \norma{\hat{\mathcal{U}}^h - \check{\mathcal{U}}^h} _{\LL1
          ([0,t] \times \partial\mathcal{X};\reali)}\right)
      \\
      & \quad + e^{(P_1 + P_2 M) t}
      \left(\norma{\hat{\mathcal{Q}}^h-\check{\mathcal{Q}}^h} _{\LL1
          ([0,t] \times \mathcal{X};\reali)} + K \norma{\hat{\mathcal
            P}^h - \check{\mathcal P}^h} _{\LL1([0, t];
          \LL\infty(\mathcal{X}; \reali))}\right),
    \end{split}
  \end{equation}
  where, by~\ref{hyp:g_b} and~\ref{ip:(ub)},
  \begin{equation*}
    \begin{split}
      K & = \norma{\hat u_{o,h} }_{\LL1 (\mathcal{X}; \reali)} +
      \norma{v}_{\LL\infty ([0,t]\times\mathcal{X};\reali^{n+m})}
      \norma{\check{\mathcal{U}}^h} _{\LL1 ([0,t] \times
        \partial\mathcal{X};\reali)} +\norma{\check{\mathcal{Q}}^h}
      _{\LL1 ([0,t] \times \mathcal{X};\reali)}
      \\
      & \le M + \norma{v}_{\LL\infty
        ([0,t]\times\mathcal{X};\reali^{n+m})}
      \norma{B}_{\LL1(\partial \mathcal{X}; \reali)} \left(M +
        1\right) \, T + Q_1 T M + \norma{Q_2}_{\LL1(\mathcal{X};
        \reali)} T M + Q_3 T M^2.
    \end{split}
  \end{equation*}
  Using~\ref{ip:(ub)}, \ref{hyp:g_b} and~\ref{hyp:g_a}, we have:
  \begin{eqnarray*}
    \norma{\hat{\mathcal{U}}^h - \check{\mathcal{U}}^h}
    _{\LL1 ([0,t]\times \partial\mathcal{X}; \reali)}
    & \leq
    & \norma{B}_{\LL1 (\partial\mathcal{X};\reali)} \,
      \norma{\hat u - \check u}_{\LL1 ([0,t]\times\mathcal{X}; \reali^k)},
    \\
    \norma{\hat{\mathcal{Q}}^h - \check{\mathcal{Q}}^h}
    _{\LL1 ([0,t]\times \mathcal{X}; \reali)}
    & \leq
    & Q_1 \norma{\hat u^h - \check u^h}_{\LL1 ([0,t]\times\mathcal{X}; \reali)}
    \\
    &
    &  +
      Q_3
      \left(
      \norma{\hat u}_{\LL\infty ([0,t];\LL1 (\mathcal{X};\reali^k))}
      +
      \norma{\check u}_{\LL\infty ([0,t];\LL1 (\mathcal{X};\reali^k))}
      \right)
    \\
    &
    & \qquad \times
      \norma{\hat u - \check u}_{\LL1 ([0,t]\times \mathcal{X};\reali^k)}
    \\
    & \leq
    & Q_1\,  \norma{\hat u^h - \check u^h}_{\LL1 ([0,t]\times\mathcal{X}; \reali)}
      +
      2 M Q_3 \,
      \norma{\hat u - \check u}_{\LL1 ([0,t]\times \mathcal{X};\reali^k)},
    \\
    \norma{\hat{\mathcal{P}}^h - \check{\mathcal{P}}^h}
    _{\LL1 ([0,t]; \LL\infty\left(\mathcal{X}; \reali\right))}
    & \le
    & \int_0^t \norma{\hat{\mathcal{P}}^h(s) - \check{\mathcal{P}}^h(s)}
      _{\LL\infty(\mathcal{X}; \reali)} \dd s
    \\
    & \le
    & \int_0^t
      \norma{
      p^h\left(s, \cdot, \hat u (s)\right)
      -
      p^h\left(s, \cdot, \check u (s)\right)}_{\LL\infty(\mathcal{X}; \reali)}
      \dd s
    \\
    & \le
    & P_2 \int_0^t
      \norma{\hat u(s) - \check u(s)}_{\LL1(\mathcal{X}; \reali^k)} \dd s
    \\
    & =
    & P_2 \, \norma{\hat u - \check u}
      _{\LL1([0, t] \times \mathcal{X}; \reali^k)}.
  \end{eqnarray*}
  Inserting these estimates into~\eqref{eq:stab-est-proof} we deduce
  that
  \begin{align*}
    & \norma{\hat{u}^h(t) - \check{u}^h(t)}_{\LL1 (\mathcal{X};\reali)}
    \\
    \le \
    & e^{(P_1 + P_2 M) t}
      \norma{\hat u_{o} - \check u_{o}}_{\LL1 (\mathcal{X}; \reali^k)}
    \\
    & + e^{(P_1 + P_2 M) t}\! \left(
      \norma{v}_{\LL\infty ([0,t]\times\mathcal{X};\reali^{n+m})}
      \norma{B}_{\LL1 (\partial\mathcal{X}; \reali)}+
      Q_1 + 2M Q_3 + K P_2 \right) \!
      \norma{\hat u - \check u}_{\LL1 ([0,t]\times\mathcal{X}; \reali^k)}.
  \end{align*}
  Sum over $h = 1, \ldots, k$ and use Gronwall Lemma to
  prove~\ref{thm:it:ID}, completing the proof.
\end{proofof}

\begin{proofof}{Corollary~\ref{cor:piu}}
  For every $w \in X$, with $X$ as in~\eqref{eq:X_M}, define
  $u = \mathcal T w$ as the image of $w$ through the operator
  $\mathcal T$, defined in~\eqref{eq:operator-T}. By~\eqref{eq:8}, we
  deduce that $u^h(t, x) \ge 0$ for a.e.~$x \in \mathcal{X}$.  This
  implies that also the unique fixed point of the operator
  $\mathcal T$ has the same property, thus~\eqref{eq:positivity}
  holds.
\end{proofof}

\begin{proofof}{Corollary~\ref{cor:assumpt-defin-result}}
  By Theorem~\ref{thm:main}, we know that there exists a solution
  $u \in \C0 ([0,T]; \LL1 (\mathcal{X}; \reali^k))$ and that this
  solution can be uniquely extended beyond time $T$ as long as
  $\norma{u (T)}_{\LL1 (\mathcal{X}; \reali^k)}$ is bounded. By
  Corollary~\ref{cor:piu},
  $\norma{u (t)}_{\LL1 (\mathcal{X}; \reali^k)} = \sum_{h=1}^k
  \int_{\mathcal{X}} u^h (t,x) \dd{x}$. Using~\eqref{eq:32}, the
  Divergence Theorem and~\ref{ip:(ub)}, we have
  \begin{align*}
    & \dfrac{\dd{~}}{\dd{t}} \norma{u (t)}_{\LL1 (\mathcal{X}; \reali^k)}
    \\
    = \
    & \dfrac{\dd{~}}{\dd{t}}
      \sum_{h=1}^k
      \int_{\mathcal{X}} u^h (t,x) \dd{x}
    \\
    = \
    & \sum_{h=1}^k
      \int_{\mathcal{X}}
      \left(
      p^h\left(t,x, u (t)\right) u (t)
      +
      q^h\left(t,x, u (t,x), u (t)\right)
      \right)
      \dd{x}
      +
      \sum_{h=1}^k \int_{\partial\mathcal{X}}
      u_b^h \left(t, \xi, u (t)\right) \dd\xi
    \\
    \leq \
    & \int_{\mathcal{X}} \left(C_1 (t,x)  + C_2 (t)\, \sum_{h=1}^k u^h (t,x)\right) \dd{x}
      + \int_{\partial\mathcal{X}} B (\xi)
      \left(k + \norma{u (t)}_{\LL1 (\mathcal{X}; \reali^k)}\right) \dd\xi
    \\
    = \
    & \left(\norma{C_1}_{\LL\infty ([0,t]; \LL1(\mathcal{X}; \reali))}
      {+}
      k \, \norma{B}_{\LL1 (\partial\mathcal{X}; \reali)}\right)
      +
      \left(
      \norma{C_2}_{\LL\infty ([0,t]; \reali)}
      {+}
      \norma{B}_{\LL1 (\partial\mathcal{X}; \reali)}
      \right)
      \norma{u (t)}_{\LL1 (\mathcal{X}; \reali^k)}
  \end{align*}
  and usual ODE estimates ensure that
  $\norma{u (t)}_{\LL1 (\mathcal{X}; \reali^k)}$ is bounded on bounded
  intervals.
\end{proofof}

\begin{proofof}{Theorem~\ref{thm:stab}}
  We divide the proof in several steps.
  \paragraph{Theorem~\ref{thm:main} Applies.}
  We first check that the assumptions of Theorem~\ref{thm:main} hold.

  \textbf{\ref{hyp:g_a} holds.} Fix
  $h \in \left\{1, \ldots, k\right\}$, $t \in I$ and
  $x \in \mathcal X$.  If $w \in \LL1(\mathcal X; \reali^k)$, then
  \begin{align*}
    \modulo{p^h\left(t, x, w\right)}
    & \le \bar P_1 + \bar P_2
      \norma{\int_{\mathcal X}{\mathcal K}^h_p\left(t, x, x'\right) w(x')
      \dd x'}
    \\
    & \le \bar P_1 + \bar P_2 \norma{{\mathcal K}^h_p}_{\LL\infty([0,t]\times
      \mathcal X^2; \reali^{k_p k})}
      \norma{w}_{\LL1(\mathcal X; \reali^k)}.
  \end{align*}
  If $w_1, w_2 \in \LL1(\mathcal X; \reali^k)$, then
  \begin{align*}
    \modulo{p^h\left(t, x, w_1\right) - p^h\left(t, x, w_2\right)}
    & \le \bar P_2
      \norma{\int_{\mathcal X}\modulo{{\mathcal K}^h_p\left(t, x, x'\right)}
      \modulo{w_1(x') - w_2(x')}
      \dd x'}
    \\
    & \le \bar P_2 \norma{{\mathcal K}^h_p}_{\LL\infty([0,t]\times
      \mathcal X^2; \reali^{k_p k})} \,
      \norma{w_1 - w_2}_{\LL1(\mathcal X; \reali^k)}.
  \end{align*}
  Therefore \ref{hyp:g_a} holds with $P_1 = \bar P_1$ and
  $P_2 = \bar P_2 \norma{{\mathcal K}^h_p}_{\LL\infty([0,t]\times
    \mathcal X^2; \reali^{k_p k})}$.

  \textbf{\ref{hyp:g_b} holds.}  Fix
  $h \in \left\{1, \ldots, k\right\}$, $t \in I$ and
  $x \in \mathcal X$.  If $u \in \reali^k$ and
  $w \in \LL1(\mathcal X; \reali^k)$, then
  \begin{align*}
    \modulo{q^h\left(t, x, u, w\right)}
    & = \modulo{Q^h\left(t, x, u,
      \int_{\mathcal X}{\mathcal K}^h_q\left(t, x, x'\right) w(x')
      \dd x'\right)}
    \\
    &
      \le \bar Q_1 \norma{u} {+} \bar Q_2(x)
      \norma{\int_{\mathcal X} \! {\mathcal K}^h_q\left(t, x, x'\right) w(x')
      \dd x'}
        {+} \bar Q_3 \norma{u}
        \norma{\int_{\mathcal X} \! {\mathcal K}^h_q\left(t, x, x'\right) w(x') \dd x'}
    \\
    & \le \bar Q_1 \norma{u}
      + \left(\bar Q_2(x) + \bar Q_3 \norma{u}\right)
      \norma{{\mathcal K}^h_q}_{\LL\infty([0,t]\times \mathcal X^2; \reali^{k_q k})}
      \norma{w}_{\LL1(\mathcal X; \reali^k)}.
  \end{align*}
  If $u_1, u_2 \in \reali^k$ and
  $w_1, w_2 \in \LL1(\mathcal X; \reali^k)$, then
  \begin{align*}
    & \modulo{q^h\left(t, x, u_1, w_1\right) - q^h\left(t, x, u_2, w_2\right)}
    \\
    & \le \bar Q_1 \norma{u_1 - u_2}
      + \bar Q_3 \norma{{\mathcal K}^h_u}
      _{\LL\infty([0,t] \times \mathcal X^2; \reali^{k_q k})}
      \norma{w_1}_{\LL1(\mathcal X; \reali^k)} \norma{u_1 - u_2}
    \\
    & \quad + \bar Q_3 \norma{u_2} \norma{{\mathcal K}^h_u}
      _{\LL\infty([0,t] \times \mathcal X^2; \reali^{k_q k})}
      \norma{w_1 - w_2}_{\LL1(\mathcal X; \reali^k)}.
  \end{align*}
  Therefore, condition~\ref{hyp:g_b} holds with $Q_1 = \bar Q_1$,
  $Q_2(x) = \bar Q_2(x) \norma{{\mathcal
      K}^h_q}_{\LL\infty([0,t]\times \mathcal X^2; \reali^{k_q k})}$,
  and
  $Q_3 = \bar Q_3(x) \norma{{\mathcal K}^h_q}_{\LL\infty([0,t]\times
    \mathcal X^2; \reali^{k_q k})}$. \ref{hyp:g-positivity} is
  straightforward.

  \textbf{\ref{ip:(ub)} holds:}
  \begin{align*}
    \modulo{u_b^h (t,\xi,w)}
    \leq \
    &\bar B (\xi)
      \left(
      1
      +
      \norma{\int_{\mathcal{X}} {\mathcal K}^h_u (t,\xi,x') \, w (x') \dd{x'}}
      \right)
    \\
    \leq \
    &\bar B (\xi)
      \left(
      1
      +
      \norma{{\mathcal K}^h_u}_{\LL\infty ([0,t]\times\partial\mathcal{X} \times \mathcal{X}; \reali^{k_u k})}  \, \norma{w}_{\LL1 (\mathcal{X};\reali^k)}
      \right) \,.
    \\
    \modulo{u_b^h (t,\xi,w)-u_b^h (t,\xi,w')}
    \leq \
    & \bar B (\xi)
      \norma{{\mathcal K}^h_u}_{\LL\infty ([0,t]\times\partial\mathcal{X} \times \mathcal{X}; \reali^{k_u k})}  \, \norma{w-w'}_{\LL1 (\mathcal{X};\reali^k)}
  \end{align*}
  so~\ref{ip:(ub)} holds with
  $B (\xi) = \bar B (\xi) \left(1 + \norma{{\mathcal K}_u}_{\LL\infty
      ([0,t]\times\partial\mathcal{X} \times \mathcal{X}; \reali^{k_u
        k^2})}\right)$.  Clearly, also~\ref{ip:(ub-2)} holds.

  \paragraph{Stability Estimates.}We now pass to the stability
  estimates. In each of the following cases, we keep $t \in I$ fixed
  and $h \in \{1, \ldots, k\}$. Define
  \begin{equation}
    \label{eq:23}
    \begin{array}{@{}r@{\,}c@{\,}l@{\quad}%
      r@{\,}c@{\,}l@{\quad}r@{\,}c@{\,}l@{}}
      \hat{\mathcal{U}}_b^h (t,\xi)
      & =
      & \hat{u}_b^h\left(t, \xi, \hat{u} (t)\right) ,
      & \hat{\mathcal{Q}}^h (t,x)
      & =
      & \hat{q}^h \left(t,x,\hat{u} (t,x),\hat{u} (t)\right),
      & \hat{\mathcal{P}}^h (t,x)
      & =
      & \hat{p}^h \left(t,x,\hat{u} (t)\right),
      \\
      \check{\mathcal{U}}_b^h (t,\xi)
      & =
      & \check{u}_b^h\left(t, \xi, \check{u} (t)\right) ,
      &  \check{\mathcal{Q}}^h (t,x)
      & =
      & \check{q}^h \left(t,x,\check{u}(t,x),\check{u}(t)\right),
      & \check{\mathcal{P}}^h (t,x)
      & =
      & \check{p}^h \left(t,x,\check{u}(t)\right) .
    \end{array}
  \end{equation}
  In order to use~\Cref{lem:stability-linear-system}, compute
  preliminarily
  \begin{displaymath}
    \mathcal{P} (t)
    =
    \exp\left(t \, \max \left\{
        \norma{\hat{\mathcal{P}}^h}_{\LL\infty ([0,t]\times\mathcal{X}; \reali)}
        \,,\;
        \norma{\check{\mathcal{P}}^h}_{\LL\infty ([0,t]\times\mathcal{X}; \reali)}
      \right\}
    \right)
    \leq
    \exp\left(t \, (P_1 + P_2 M)\right) \,,
  \end{displaymath}
  where $M$ is an upper bound for the $\LL\infty$ in time and $\LL1$
  in space norms of both solutions. Therefore,
  \Cref{lem:stability-linear-system} implies that
  \begin{equation}
    \label{eq:stab-estimate-proof}
    \begin{split}
      & \norma{\hat{u}^h (t) - \check{u}^h (t)}_{\LL1 (\mathcal{X};
        \reali)}
      \\
      \leq\, & \mathcal{P} (t) \, \norma{v}_{\LL\infty
        ([0,t]\times\mathcal{X}; \reali^{n+m})} \,
      \norma{\hat{\mathcal{U}}^h_b- \check{\mathcal{U}}^h_b} _{\LL1
        ([0,t]\times \partial\mathcal{X};\reali)} + \mathcal{P} (t) \,
      \norma{\hat{\mathcal{Q}}^h - \check{\mathcal{Q}}^h} _{\LL1
        ([0,t]\times\mathcal{X};\reali)}
      \\
      & + \mathcal{P} (t) \, \left( \norma{u_o}_{\LL1 (\mathcal{X};
          \reali^k)} +\norma{\check{\mathcal{Q}}^h}_{\LL1 ([0,t]
          \times \mathcal{X}; \reali)} \right)
      \norma{\hat{\mathcal{P}}^h - \check{\mathcal{P}}^h} _{\LL1
        ([0,t]; \LL\infty(\mathcal{X}; \reali))}
      \\
      & + \mathcal{P} (t) \, \norma{v}_{\LL\infty
        ([0,t]\times\mathcal{X}; \reali^{k\times(n+m)})}\,
      \norma{\check{\mathcal{U}}_{b}} _{\LL1
        ([0,t]\times\partial\mathcal{X};\reali^k)}
      \norma{\hat{\mathcal{P}}^h - \check{\mathcal{P}}^h} _{\LL1
        ([0,t]; \LL\infty(\mathcal{X}; \reali))}.
    \end{split}
  \end{equation}
  Then, we estimate the terms in~\eqref{eq:stab-estimate-proof}.
  Using~\ref{ip:(ub)} and~\eqref{eq:23} we deduce that
  \begin{eqnarray*}
    &
    & \norma{\hat{\mathcal{U}}^h_b- \check{\mathcal{U}}^h_b}_{\LL1 ([0,t]\times  \partial\mathcal{X};\reali)}
    \\
    & =
    & \int_0^t \int_{\partial\mathcal{X}}
      \modulo{\hat{u}_b^h\left(\tau, \xi, \hat{u} (\tau)\right)
      -
      \check{u}_b^h\left(\tau, \xi, \check{u} (\tau)\right)
      }
      \dd\xi \dd{\tau}
    \\
    & \leq
    & \int_0^t \int_{\partial\mathcal{X}}
      \modulo{
      \hat{u}_b^h\left(\tau, \xi, \hat{u} (\tau)\right)
      -
      \hat{u}_b^h\left(\tau, \xi, \check{u} (\tau)\right)
      }
      \dd\xi \dd{\tau}
    \\
    &
    & +
      \int_0^t \int_{\partial\mathcal{X}}
      \modulo{
      \hat{u}_b^h\left(\tau, \xi, \check{u} (\tau)\right)
      -
      \check{u}_b^h\left(\tau, \xi, \check{u} (\tau)\right)
      }
      \dd\xi \dd{\tau}
    \\
    & \le
    & \norma{B}_{\LL1 (\partial\mathcal{X};\reali)} \,
      \norma{\hat u - \check u}_{\LL1 ([0,t]\times \mathcal{X}; \reali^k)}
    \\
    &
    & + \int_0^t \int_{\partial\mathcal{X}}
      \modulo{
      \hat U^h_b \left(
      \tau, \xi,
      \int_{\mathcal{X}} \hat{\mathcal{K}}^h_u (\tau,\xi,x') \, \check u (\tau,x') \dd{x'}
      \right)
      -
      \hat U^h_b \left(
      \tau, \xi,
      \int_{\mathcal{X}} \check{\mathcal{K}}^h_u (\tau,\xi,x') \, \check u (\tau,x') \dd{x'}
      \right)} \dd \xi\, \dd \tau
    \\
    &
    & + \int_0^t \int_{\partial\mathcal{X}}
      \modulo{
      \hat U^h_b \left(
      \tau, \xi,
      \int_{\mathcal{X}} \check{\mathcal{K}}^h_u (\tau,\xi,x') \, \check u (\tau,x') \dd{x'}
      \right)
      -
      \check U^h_b \left(
      \tau, \xi,
      \int_{\mathcal{X}} \check{\mathcal{K}}^h_u (\tau,\xi,x') \, \check u (\tau,x') \dd{x'}
      \right)} \dd \xi\, \dd \tau
    \\
    & \le
    & \norma{B}_{\LL1 (\partial\mathcal{X};\reali)} \,
      \norma{\hat u - \check u}_{\LL1 ([0,t]\times \mathcal{X}; \reali^k)}
    \\
    &
    & +
      \int_0^t \int_{\partial\mathcal{X}} \bar B (\xi)
      \norma{\hat{\mathcal{K}}^h_u - \check{\mathcal{K}}^h_u}_{\LL\infty ([0,t]\times\partial\mathcal{X}\times\mathcal{X}; \reali^{k_u k})}
      \norma{\check u (\tau)}_{\LL1 (\mathcal{X}; \reali^k)} \dd\xi \dd\tau
    \\
    &
    & +
      \norma{\hat U_b^h - \check U_b^h}_{\LL1 ([0,t]\times \partial\mathcal{X}; \LL\infty (\reali^{k_u};\reali))}
    \\
    \\
    & \le
    & \norma{B}_{\LL1 (\partial\mathcal{X};\reali)} \,
      \norma{\hat u - \check u}_{\LL1 ([0,t]\times \mathcal{X}; \reali^k)}
      +
      \norma{\bar B}_{\LL1 (\partial\mathcal{X}; \reali)}
      \norma{\hat{\mathcal{K}}^h_u - \check{\mathcal{K}}^h_u}_{\LL\infty ([0,t]\times\partial\mathcal{X}\times\mathcal{X}; \reali^{k_u k})}
      \norma{\check u }_{\LL1 ([0,t]\times\mathcal{X}; \reali^k)}
    \\
    &
    & +
      \norma{\hat U_b^h - \check U_b^h}_{\LL1 ([0,t]\times \partial\mathcal{X}; \LL\infty (\reali^{k_u};\reali))} \,.
  \end{eqnarray*}
  Using~\ref{hyp:g_b} we deduce that
  \begin{eqnarray*}
    &
    & \norma{\hat{\mathcal{Q}}^h - \check{\mathcal{Q}}^h}
      _{\LL1 ([0,t]\times\mathcal{X};\reali)}
    \\
    & \leq
    & \int_0^t \int_{\mathcal{X}}
      \modulo{
      \hat{q}^h \left(\tau,x,\hat{u} (\tau,x),\hat{u} (\tau)\right)
      -
      \hat{q}^h \left(\tau,x,\check{u} (\tau,x),\check{u} (\tau)\right)
      }
      \dd{x} \dd{\tau}
    \\
    &
    & +
      \int_0^t \int_{\mathcal{X}}
      \modulo{
      \hat{q}^h \left(\tau,x,\check{u} (\tau,x),\check{u} (\tau)\right)
      -
      \check{q}^h \left(\tau,x,\check{u} (\tau,x),\check{u} (\tau)\right)
      }
      \dd{x} \dd{\tau}
    \\
    & \le
    & Q_1 \int_0^t \norma{\hat u(\tau) - \check u(\tau)}
      _{\LL1(\mathcal X; \reali^k)} \dd \tau
          + Q_3 \int_0^t \norma{\hat u(\tau)}_{\LL1(\mathcal X; \reali^k)}
          \int_{\mathcal X} \norma{\hat u(\tau, x) - \check u(\tau, x)} \dd x\,
          \dd \tau
    \\
    &
    & + Q_3 \int_0^t \norma{\hat u(\tau) - \check u(\tau)}
      _{\LL1(\mathcal X; \reali^k)}
      \int_{\mathcal X} \norma{\check u(\tau, x)} \dd x\, \dd \tau
    \\
    &
    & + \int_0^t \int_{\mathcal{X}}
      \left|
      \hat Q^h\left(\tau, x, \check u (\tau,x), \int_{\mathcal X}
      \hat{\mathcal K}^h_q\left(\tau, x, x'\right)
      \check u\left(\tau, x'\right) \dd x'\right)
      \right.
    \\
    &
    & \qquad\qquad\quad \left.
      -
      \check Q^h\left(\tau, x, \check u (\tau,x), \int_{\mathcal X}
      \check{\mathcal K}^h_q\left(\tau, x, x'\right)
      \check u\left(\tau, x'\right) \dd x'\right)
      \right|
      \dd{x} \dd\tau
    \\
    & \leq
    & \left(Q_1
      +
      Q_3
      \left(
      \norma{\hat u}_{\LL\infty ([0,t];\LL1 (\mathcal{X}; \reali^k))}
      +
      \norma{\check u}_{\LL\infty ([0,t];\LL1 (\mathcal{X}; \reali^k))}
      \right)
      \right)
      \int_0^t \norma{\hat u(\tau) - \check u(\tau)}
      _{\LL1(\mathcal{X}; \reali^k)} \dd \tau
    \\
    &
    & +
      \int_0^t \int_{\mathcal{X}}
      \sup_{\eta \in \reali^{k_q}}
      \modulo{
      \hat Q^h\left(\tau, x,\check u (\tau,x), \eta\right)
      -
      \check Q^h\left(\tau, x, \check u (\tau,x), \eta\right)
      }
      \dd{x} \dd\tau
    \\
    &
    & + \bar Q_3 \int_0^t \int_\mathcal{X}
      \norma{\check u(\tau, x)}
      \norma{\int_{\mathcal X}
      \left(
      \hat{\mathcal{K}}^h_q(\tau, x, x') - \check{\mathcal{K}}^h_q
      (\tau, x, x')\right)
      \check u\left(\tau, x'\right) \dd x'}\, \dd x\, \dd\tau
    \\
    & \leq
    & \left(Q_1
      +
      Q_3
      \left(
      \norma{\hat u}_{\LL\infty ([0,t];\LL1 (\mathcal{X}; \reali^k))}
      +
      \norma{\check u}_{\LL\infty ([0,t];\LL1 (\mathcal{X}; \reali^k))}
      \right)
      \right)
      \norma{\hat u - \check u}_{\LL1([0,t]\times\mathcal{X}; \reali^k)}
    \\
    &
    & +
      \norma{\hat Q^h - \check Q^h}_{\LL1 ([0,t]\times\mathcal{X}; \LL\infty (\reali^k\times\reali^{k_q}; \reali
      ))}
      + \int_0^t
      \norma{\check u\left(\tau\right)}^2_{\LL1(\mathcal X; \reali^k)}
      \dd\tau
      \norma{\hat{\mathcal K}^h_q - \check{\mathcal K}^h_q}
      _{\LL\infty([0,t] \times \mathcal X^2; \reali^{k_q})}.
  \end{eqnarray*}
  Using~\ref{hyp:g_a}, we have
  \begin{eqnarray*}
    &
    & \norma{\hat{\mathcal{P}}^h - \check{\mathcal{P}}^h}
      _{\LL1 ([0,t]; \LL\infty(\mathcal{X}; \reali))}
    \\
    & \le
    & \int_0^t
      \sup_{x \in \mathcal{X}}
      \modulo{
      \hat{p}^h\left(\tau,x,\hat u (\tau)\right)
      -
      \hat{p}^h\left(\tau,x,\check u (\tau)\right)
      }
      \dd{\tau}
      + \int_0^t
      \sup_{x \in \mathcal{X}}
      \modulo{
      \hat{p}^h\left(\tau,x,\check u (\tau)\right)
      -
      \check{p}^h\left(\tau,x,\check u (\tau)\right)
      }
      \dd{\tau}
    \\
    & \le
    & P_2 \int_0^t \norma{\hat u(\tau) - \check u(\tau)}
      _{\LL1(\mathcal X; \reali^k)}\dd \tau
    \\
    &
    & + \int_0^t
      \sup_{x \in \mathcal{X}}
      \modulo{
      \hat{P}^h
      \left(
      \tau, x,
      \int_{\mathcal{X}} \hat{\mathcal{K}}^h_p (\tau, x, x')
      \check{u} (\tau, x') \dd{x'}
      \right)
      -
      \check{P}^h
      \left(
      \tau, x,
      \int_{\mathcal{X}} \hat{\mathcal{K}}^h_p (\tau, x, x')
      \check{u} (\tau, x') \dd{x'}
      \right)
      }
      \dd{\tau}
    \\
    &
    & + \int_0^t
      \sup_{x \in \mathcal{X}}
      \modulo{
      \check{P}^h
      \left(
      \tau, x,
      \int_{\mathcal{X}} \hat{\mathcal{K}}^h_p (\tau, x, x')
      \check{u} (\tau, x') \dd{x'}
      \right)
      -
      \check{P}^h
      \left(
      \tau, x,
      \int_{\mathcal{X}} \check{\mathcal{K}}^h_p (\tau, x, x')
      \check{u} (\tau, x') \dd{x'}
      \right)
      }
      \dd{\tau}
    \\
    & \le
    & P_2 \int_0^t \norma{\hat u(\tau) - \check u(\tau)}
      _{\LL1(\mathcal X; \reali^k)}\dd \tau
      + t \norma{\hat P^h - \check P^h}_{\LL\infty\left([0,t] \times \mathcal X
      \times \reali^{k_p}; \reali\right)}
    \\
    &
    & + \bar P_2 \int_0^t
      \sup_{x \in \mathcal{X}}
      \int_{\mathcal X}
      \modulo{
      \hat{\mathcal{K}}^h_p (\tau, x, x')
      -
      \check{\mathcal{K}}^h_p (\tau, x, x')}
      \modulo{\check{u} (\tau, x')} \dd x'
      \dd{\tau}
    \\
    & \le
    & P_2 \int_0^t \norma{\hat u(\tau) - \check u(\tau)}
      _{\LL1(\mathcal X; \reali^k)}\dd \tau
      + t \norma{\hat P^h - \check P^h}_{\LL\infty\left([0,t] \times \mathcal X
      \times \reali^{k_p}; \reali\right)}
    \\
    &
    & + \bar P_2 \int_0^t \norma{\check u(\tau)}
      _{\LL1(\mathcal X; \reali^k)} \dd \tau
      \norma{
      \hat{\mathcal{K}}^h_p
      -
      \check{\mathcal{K}}^h_p}_{\LL\infty\left([0,t]\times \mathcal X^2; \reali^{k_p k}
      \right)}.
  \end{eqnarray*}
  The above estimate, duly inserted in~\eqref{eq:stab-estimate-proof}
  and followed by a standard application of Gronwall Lemma, completes
  the proof.
\end{proofof}

\begin{proofof}{Proposition~\ref{prop:spreading-pandemic}}
  Checking~\ref{ip:(v)} and~\ref{ip:(u0)} is immediate. It is
  sufficient to verify that the assumptions of Theorem~\ref{thm:stab}
  hold. It is immediate to check that~\ref{stab:it:1} holds with
  $\bar P_1 = \max \{\norma{\mu_S}, \norma{\mu_I} +
  \norma{\kappa+\theta}, \norma{\mu_H} + \norma{\eta},
  \norma{\mu_R}\}$ (all norms being in
  ${\LL\infty (\mathcal{I}\times\reali_+\times\reali^2; \reali)}$),
  $\bar P_2 = 1$, thanks to $\rho \in
  \LL\infty$. Concerning~\ref{stab:it:2}, choose
  $\bar Q_1 = \max\{\norma{\kappa}, \norma{\eta+\theta}\}$,
  $\bar Q_2 = 0$, $\bar Q_3 = 1$ and use $\rho \in
  \LL\infty$. Finally, \ref{stab:it:3} holds with
  $\bar B (\xi) = \sup_I\norma{S_b (t)}_{\LL\infty
    (\mathcal{X},\reali)}$.

  Positivity is immediate. To apply
  Corollary~\ref{cor:assumpt-defin-result}, simply set $C_1 \equiv 0$
  and $C_2 \equiv 0$.

  To obtain an $\L\infty$ bound, note first that since
  $I \in \C0\left(\mathcal{I};\L1 (\reali_+ \times \reali^2; \reali)
  \right)$, the integral in~\eqref{eq:6} is bounded on any bounded
  time interval. Hence, a repeated application of~\eqref{eq:5} in
  Lemma~\ref{lem:L1inf} yields the boundedness of $S$, $I$, $H$ and
  $R$ on any bounded interval. Uniqueness then follows
  from~\ref{thm:it:uni}.
\end{proofof}

\begin{proofof}{Proposition~\ref{prop:an-age-phen}}
  Assumptions~\ref{ip:(v)} and~\ref{ip:(u0)} trivially
  hold. Condition~\ref{hyp:g_a} holds with
  $P_1 = \norma{d}_{\LL\infty (\reali^n; \reali)} / \eps$ and
  $P_2= 1 /\eps$. Verifying~\ref{hyp:g_b} is straightforward. To prove
  that~\ref{ip:(ub)} holds, compute for $y \in \reali^n$ with
  $\norma {y}> r$:
  \begin{eqnarray*}
    \modulo{u_b (t,y,w)}
    & =
    & \modulo{\dfrac{1}{A (a=0,y) \, \varepsilon^n}
      \int_{\reali_+} \! \int_{\reali^n}
      \mathcal{M} \! \left(\frac{y'-y}{\varepsilon}\right) \,
      b (a',y') \,
      w(a', y')
      \dd{a'} \dd{y'}}
    \\
    & \leq
    & \dfrac{1}{\epsilon^n \; \inf A}\;
      \modulo{\int_{\reali_+} \! \int_{\reali^n}
      \mathcal{M} \! \left(\frac{y'-y}{\varepsilon}\right) \,
      b (a',y') \,
      w(a', y')
      \dd{a'} \dd{y'}}
    \\
    & \leq
    & \dfrac{1}{\epsilon^n\; \inf A}\;
      \int_{\reali_+} \! \int_{\reali^n}
      \modulo{\mathcal{M} \! \left(\frac{y'-y}{\varepsilon}\right)} \,
      \left(\sup_{\modulo{y'-y} <r} \modulo{b (a',y')}\right) \,
      \modulo{w(a', y')}
      \dd{a'} \dd{y'}
    \\
    & \leq
    & \dfrac{1}{\epsilon^n \; \inf A} \;
      \norma{\mathcal{M}}_{\LL\infty (\reali^n; \reali)} \;
      \dfrac{1}{\left(1+\norma{y}-r\right)^{n+1}}
      \int_{\reali_+} \! \int_{\reali^n}
      \modulo{w(a', y')}
      \dd{a'} \dd{y'}
  \end{eqnarray*}
  proving the first requirement in~\ref{ip:(ub)}. Lipschitz continuity
  is proved by the same procedure.

  The assumptions on the signs of data and parameters allow to apply
  Corollary~\ref{cor:piu} and ensure that also~\eqref{eq:36} holds.
\end{proofof}

\section*{Acknowledgments}

The authors were partly supported by the GNAMPA 2022 project
\emph{Evolution Equations:Well Posedness, Control and Applications}.

{\footnotesize

  \bibliographystyle{abbrv}

  \bibliography{CGMR_revised.bib}

}
\end{document}